\documentclass[12pt,reqno]{amsart}
\usepackage[left=3cm,top=2cm,right=3cm,bottom=2cm]{geometry}
\usepackage{amssymb}
\usepackage{amsbsy}
\usepackage{epsfig}
\usepackage{tikz}
\usepackage[english]{babel}    
\usepackage{enumitem}
\usepackage{multicol}
\usepackage{blkarray}
\usepackage{tikz}
\usepackage{mathtools}
\usepackage{amsthm}
\newtheorem{teore}{Theorem}[section]
\newtheorem{defn}[teore]{Definition}
\newtheorem{lemat}[teore]{Lemma}

\newtheorem{prop}[teore]{Proposition}
\newtheorem{remark}[teore]{Remark}
\begin{document}
\DeclarePairedDelimiter\ceil{\lceil}{\rceil}
\DeclarePairedDelimiter\floor{\lfloor}{\rfloor}
\title{Null Decomposition of Unicyclic Graphs}
\author[L. E. Allem]{L. Emilio Allem}\email{emilio.allem@ufrgs.br}
\address{UFRGS - Universidade Federal do Rio Grande do Sul, Instituto de Matem\'atica, Porto Alegre, Brazil}
\author[D. A. Jaume]{Daniel A. Jaume}\email{djaume@unsl.edu.ar}
\address{Universidad Nacional de San Luis, Departamento de Matem\'aticas, San Luis, Argentina}
\author[G. Molina]{Gonzalo Molina}\email{lgmolina@unsl.edu.ar}
\address{Universidad Nacional de San Luis, Departamento de Matem\'aticas, San Luis, Argentina}
\author[M. M. Toledo]{Maikon M. Toledo}\email{maikon.toledo@ufrgs.br}
\address{UFRGS - Universidade Federal do Rio Grande do Sul, Instituto de Matem\'atica, Porto Alegre, Brazil}
\author[V. Trevisan]{Vilmar Trevisan}\email{trevisan@mat.ufrgs.br }
 \address{UFRGS - Universidade Federal do Rio Grande do Sul, Instituto de Matem\'atica, Porto Alegre, Brazil}

\begin{abstract}
In this work we obtain basis for the null space of unicyclic graphs. We extend the null decomposition of trees from \cite{tree} for unicyclic graphs. As an application, we obtain closed formulas for the independence and matching numbers of unicyclic graphs just using the support of the graph.
\end{abstract}

\maketitle

\section{Introduction}

Analysing the eigenspace of graphs to obtain structural properties is a standard technique in the area of spectral graph theory. One of the most relevant examples is the pioneer work of Fiedler \cite{Fie73,Fie75}, where the eigenspace associated with the algebraic connectivity of the graph was used to obtain a vertex set decomposition. The idea was called by Nikiforov \cite{Nik13} a mathematical \emph{golden strike} due to its tremendous impact in several areas of mathematics as well as for numerous applications. One of the most important results on the Laplacian matrix is Fiedler's Monotocity Theorem \cite{lal}. This theorem relates the structure of the vertices of a graph with any eigenvector associated with the second least eigenvalue of the Laplacian matrix. In 1988, Power \cite{power} also obtained relationships between the eigenvectors and properties of the graphs.

Recently, in \cite{tree}, Jaume and Molina studied the null space of the adjacency matrix of trees and they presented a null decomposition of trees. In general, this null decomposition divides a tree into two forests (one of the forests can be empty), one composed by singular trees and the other composed by non-singular trees.
The technique used was the analysis of the \emph{support} of the tree, where the \emph{support} is defined as the subset of vertices for which at least one of its corresponding coordinates of the eigenvectors of the null space of the adjacency matrix is nonzero.

Understanding the support of a graph, in some sense,  resembles the work of Fiedler, who obtained information of the graph based on the signs of the coordinates of the eigenvector. Here, we obtain information of the graph based on whether an eigenvector coordinate is zero or nonzero. It happens that the support provides important information about the graph.

As an application, in \cite{tree}, the null decomposition was used to obtain closed formulas for two classical parameters.  The first one is the cardinality of the maximum independent set of a graph $G$, denoted by $\alpha(G)$, and the second one, the cardinality of the maximum matching  of a graph $G$, denoted by $\nu(G)$. Several mathematicians have studied $\alpha(G)$ \cite{Alon:1998:AIN:286563.286564,Frieze:1990:INR:82922.82935,SHEARER198383} and $\nu(G)$ \cite{doob,ming}.
In order to compute $\alpha(G)$,  most algorithms first find a maximum independent set and then count the number of elements in this set \cite{far,sage}. In particular this is the strategy used in SageMath and Maple softwares. In the same softwares, the problem of computing $\nu(G)$ is solved using Edmond's algorithm \cite{edmond}. It also computes a maximum matching, and then counts the number of elements in this matching.

Here, we compute  $\nu(G)$ and $\alpha(G)$ directly using linear algebra without finding maximum sets. In \cite{Maikon}, we obtained closed formulas for the independence and matching numbers of a unicyclic graph based on the support of its subtrees, here, we also present closed formulas for these parameters of a unicyclic graph, the difference is that now, we only use the support of the unicyclic graph $G$.

To give a glimpse of our results, we present the following example. First, we compute the support, core (neighborhood of the support) and $N$-vertices (the remaining vertices) of the graph $G$ of Figure \ref{glimpse}. Consider $C$ the cycle of $G$. 
$Supp(G)=\lbrace{a,b,e,f,i}\rbrace$, $Core(G)=\lbrace{c,v,g,h}\rbrace$ and $V(\mathcal{G}_N(G))=\lbrace{d,\ell,j,m,n,o,p,r,q}\rbrace$.

\usetikzlibrary{shapes,snakes}
\tikzstyle{vertex}=[circle,draw,minimum size=1pt,inner sep=1pt]
\tikzstyle{edge} = [draw,thick,-]
\tikzstyle{matched edge} = [draw,snake=zigzag,line width=1pt,-]
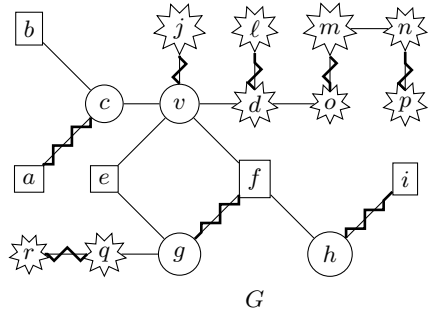
\begin{figure}[h!]
\begin{center}
\begin{scriptsize}
\begin{center}
\begin{tikzpicture}[scale=1,auto,swap]
\node[draw,circle,label=below left:] (1) at (-1,-1) {$c$};
\node[draw,rectangle,label=below left:] (2) at (-2,-2) {$a$};
\node[draw,rectangle, label=below left:] (3) at (-2,0) {$b$};
\node[draw,circle,label=below left:] (4) at (0,-1) {$v$};
\node[draw,star,star points=9,star point ratio=0.6,label=below left:] (5) at (1,-1) {$d$};
\node[draw,rectangle,label=below left:] (6) at (-1,-2) {$e$};
\node[draw,rectangle,label=below left:] (7) at (1,-2) {$f$};
\node[draw,circle,label=below left:] (8) at (0,-3) {$g$};
\node[draw,circle,label=below left:] (9) at (2,-3) {$h$};
\node[draw,rectangle,label=below left:] (10) at (3,-2) {$i$};
\node[draw,star,star points=9,star point ratio=0.6,label=below left:] (11) at (0,0) {$j$};
\node[draw,star,star points=9,star point ratio=0.6,label=below left:] (12) at (1,0) {$\ell$};
\node[draw,star,star points=9,star point ratio=0.6,label=below left:] (13) at (2,0) {$m$};
\node[draw,star,star points=9,star point ratio=0.6,label=below left:] (14) at (3,0) {$n$};
\node[draw,star,star points=9,star point ratio=0.6,label=below left:] (15) at (2,-1) {$o$};
\node[draw,star,star points=9,star point ratio=0.6,label=below left:] (16) at (3,-1) {$p$};
\node[draw,star,star points=9,star point ratio=0.6,label=below left:] (17) at (-1,-3) {$q$};
\node[draw,star,star points=9,star point ratio=0.6,label=below left:] (18) at (-2,-3) {$r$};
\node at (1,-3.6) {$G$};
\foreach \from/\to in {1/2,1/3,1/4,4/5,4/6,4/7,8/6,8/7,9/7,9/10,17/18,17/8,11/4,12/5,5/15,15/13,13/14,14/16} {
 \draw (\from) -- (\to);}
 \foreach \source / \dest in {1/2,4/11,5/12,13/15,16/14,17/18,7/8,9/10}
   \path[matched edge] (\source) -- (\dest);
\end{tikzpicture}
\caption{Unicyclic graph $G$ and its support.}\label{glimpse}
\end{center}
\end{scriptsize}
 \end{center}
\end{figure}

Notice that, $V(C)\nsubseteq{V(\mathcal{G}_N(G))}$. Therefore, by theorems \ref{thm1} and \ref{thm2}, we have that the independence and matching numbers of $G$ can be computed as follows

\begin{eqnarray*}
\alpha{(G)}&=&\vert{Supp(G)}\vert{+}\ceil*{\frac{\vert{V(\mathcal{G}_N(G))}\vert-\vert{Supp(G)\cap{Core(G)}}\vert}{2}}=5+\ceil*{\frac{9-0}{2}}=10\\
\nu{(G)}&=&\vert{Core(G)}\vert{+}\floor*{\frac{\vert{V(\mathcal{G}_N(G))}\vert-\vert{Supp(G)\cap{Core(G)}}\vert}{2}}=4+\floor*{\frac{9-0}{2}}=8.
\end{eqnarray*}

Our main goal in this paper is to extend the theory that Jaume and Molina developed for trees in \cite{tree} for unicyclic graphs. More specifically, we will obtain structural information of unicyclic graphs using their null space. That is, we extend the definition of null decomposition of trees from \cite{tree} for graphs in general. This graph decomposition divides the graph into two subgraphs, a subgraph generated by the closed neighborhood of the support and another subgraph generated by the remaining vertices. As an application we obtain closed formulas for the independence and matching numbers of unicyclic graphs. These formulas allows one  to compute independence and matching numbers using basically only the support of the graph.

Next, we give an outline of this paper. In Section \ref{sec:pre}, we present some basic notations and definitions. In the sections \ref{sec:nulltipoI} and \ref{sec:nulltipoII}, we provide linear algebra results for unicyclic graphs of types I and II, respectively. More precisely, we obtain a basis of the null space of unicyclic graphs. These results will be important for the following sections \ref{sec:suptipoI} and \ref{sec:suptipoII}, where we study the support, core and $N$-vertices of unicyclic graphs of Type $I$ and Type $II$, respectively.

As applications of this null decomposition, in section \ref{sec:nulldecom}, we obtain our main results, which are closed formulas for the matching and independence numbers of unicyclic graphs. These formulas depend on the support, core and $N$-vertices of unicyclic graphs. It turns out that these formulas are similar to the formulas obtained by Jaume and Molina \cite{tree} for trees.

\section{Preliminaries}\label{sec:pre}

Let $G=(V,E)$ be a simple graph of order $n$ and its adjacency matrix $A(G)$. Denote by $\varepsilon_{\lambda}$ the $\lambda$-eigenspace of $A(G)$, that is, $\varepsilon_{\lambda}=\lbrace{x\in{\mathbb{R}^n}: A(G)x=\lambda{x}}\rbrace$. The  $0$-eigenspace  ($\varepsilon_0$) is the focus of our work and it will be denoted by $\mathcal{N}(G)$. The nullity of a graph $G$, denoted by $\eta(G)$, is the multiplicity of the eigenvalue zero in the spectrum of $A(G)$, or, equivalently, the dimension of $\mathcal{N}(G)$. The graph $G$ is called singular if $A(G)$ is a singular matrix or $\eta(G)>0.$ Otherwise, the graph $G$ is called nonsingular.

A set $I \subset V $ of vertices of a graph $G$ is an independent set in $G$ if no two vertices in $I$ are adjacent. A maximum independent set is an independent set of maximum cardinality. The cardinality of any maximum independent set in $G$, denoted by $\alpha(G)$, is called the independence number of $G$.
$\mathcal{I}(G)$ denotes the set of all maximum independent sets of $G$.

A matching $M$ in $G$ is a set of pairwise non-adjacent edges, that is, no two edges in $M$ share a common vertex. A maximum matching is a matching of largest cardinality in $G$. The matching number of $G$, denoted by $\nu(G)$, is the size of a maximum matching in $G$. $\mathcal{M}(G)$ denotes the set of all maximum matchings of $G$.
A vertex is saturated by $M$, if it is an endpoint of one of the edges in the matching $M$. Otherwise the vertex is said non-saturated. Moreover, a matching is said to be perfect if it saturates all vertices of $G$. The set of vertices of $G$ that are not saturated by some maximum matching is known as Edmond-Gallai of $G$ and it is denoted by $EG(G)$. And, a vertex $v$ of $G$ is called matched if it is saturated by all maximum matching, that is, $v\notin EG(G)$.

The notion of support of a vector is a natural one and crucial for our purposes.

\begin{defn}\cite{tree}
Let $G$ be a graph with $n$ vertices and let $x$ be a vector of $\mathbb{R}^{n}$. The support of $x$ in $G$ is $$Supp_{G}(x)=\lbrace{v\in{V(G)}:x_{v}\neq{0}}\rbrace.$$

\noindent Let $S$ be a subset of $\mathbb{R}^{n}$. Then the support of $S$ in $G$ is $$Supp_{G}(S)=\bigcup_{x\in{S}}{Supp_{G}(x)}.$$
\end{defn}

The following result shows that in order to compute the support of an eigenspace of $A(G)$, it is enough to analyse the coordinates of the vectors of a basis of this eigenspace.
\begin{lemat}\cite{tree}\label{suppcompute}
Let $G$ be a graph and $\lambda$ an eigenvalue of $A(G)$. Let $\mathcal{B}=\lbrace{b_1,\ldots,b_k}\rbrace$ be a basis of $\varepsilon_{\lambda}$, then $Supp_G(\varepsilon_{\lambda})=Supp_G(\mathcal{B}).$
\end{lemat}

In this present paper, our concern is the support of the null space of $A(G)$, focusing in $\mathcal{N}(G)$, the $Supp_G(\mathcal{N}(G))$, which, for purposes of notation, will be denoted by $Supp(G)$. In practice, in order to compute $Supp(G)$, we use Lemma \ref{suppcompute}, we compute a basis of the null space and consider the non-null entries of the vectors in the basis to obtain the support.

We know from \cite{tree} that the support of a tree is an independent set of vertices and we state the result as lemma for future reference.

\begin{lemat}\cite{tree}
\label{tindependent}
Let $T$ be a tree, then $Supp(T)$ is an independent set of $T$.
\end{lemat}

It is very important to notice that, unlike trees, the support of unicyclic graphs is not always an independent set. For example, consider $G$ the unicyclic graph of Figure \ref{figur1}. The support of $G$ is the set $\lbrace{u,v,z,w,b,c,d,e}\rbrace $, which is not an independent set.

We will characterize unicyclic graphs whose support is an independent set (see propositions \ref{ti1}, \ref{ti4}, \ref{ti6}, \ref{ti3} and \ref{ti2}). We will see that the only unicyclic graphs whose support is not an independent set are the unicyclic graphs of Type $II$ with a cycle of length equal to $4t$, where $t\in\mathbb{N}$ (see Proposition \ref{ti5}).

\usetikzlibrary{shapes,snakes}
\tikzstyle{vertex}=[circle,draw,minimum size=1pt,inner sep=1pt]
\tikzstyle{edge} = [draw,thick,-]
\tikzstyle{matched edge} = [draw,snake=zigzag,line width=1pt,-]
\begin{figure}[h!]
\begin{scriptsize}
\begin{center}
\begin{tikzpicture}[scale=1.2,auto,swap]
\node[draw,rectangle,label=below left:] (1) at (0,0) {$u$};
\node[draw,rectangle,label=below left:] (2) at (1,0) {$v$};
\node[draw,rectangle,label=below left:] (3) at (1,1) {$w$};
\node[draw,rectangle,label=below left:] (4) at (0,1) {$z$};
\node[draw,circle,label=below left:] (5) at (-1,0) {$a$};
\node[draw,rectangle,label=below left:] (6) at (-2,1) {$b$};
\node[draw,rectangle,label=below left:] (7) at (2,1) {$c$};
\node[draw,rectangle,label=below left:] (8) at (3,1) {$d$};
\node[draw,rectangle,label=below left:] (9) at (4,1) {$e$};
\node[draw,circle,label=below left:] (10) at (2,0) {$f$};
\foreach \from/\to in {1/2,2/3,3/4,4/1,5/6,5/4,10/7,10/8,10/9,10/3}{
\draw (\from) -- (\to);}
\end{tikzpicture}
\end{center}
\end{scriptsize}
 \caption{ {\small Unicyclic graph whose support is not an independent set.}}
\label{figur1}
\end{figure}
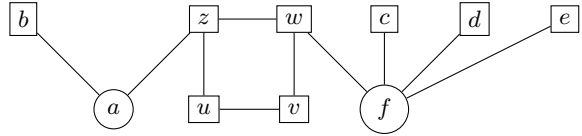

The next stated result shows that only the vertices of the support of a tree are not saturated by some maximum matching.
\begin{lemat}\cite{Maikon}
\label{t11}
Let $T$ be a tree, then $EG(T)=Supp(T)$.
\end{lemat}

Let $G$ be a unicyclic graph and let $C$ be the unique cycle of $G$. For each vertex $v\in{V(C)}$, we denote by $G\lbrace{v}\rbrace$ the induced connected subgraph of $G$ with maximum number of vertices, which contains the vertex $v$ and no other vertex of $C$. $G\lbrace{v}\rbrace$ is called the \emph{pendant tree of $G$ at $v$}.
The unicyclic graph $G$ is said to be of Type $I$ if there exists a vertex $v$ on the cycle of $G$ such that does not exist a maximum matching of $G\lbrace{v}\rbrace$ that does not saturate $v$, otherwise, $G$ is said to be of Type $II$ (for more details see \cite{Maikon}).

The next two lemmas \cite{Maikon} show that in order to verify whether a unicyclic graph is Type $I$ or $II$, it is sufficient to check whether a vertex $v$ of the cycle is or is not in the support of its pendant tree $G\lbrace{v}\rbrace$.

\begin{lemat}\cite{Maikon}
\label{t4}
A unicyclic graph $G$ is of Type $I$ if and only if there exists at least one pendant tree $G\lbrace{v}\rbrace$ such that $v\notin{Supp(G\lbrace{v}\rbrace)}$.
\end{lemat}

\begin{lemat}\cite{Maikon}
\label{c3}
A unicyclic  graph $G$ is of Type $II$ if and only if every pendant tree $G\lbrace{v}\rbrace$ is such that $v\in{Supp(G\lbrace{v}\rbrace)}$.
\end{lemat}

\begin{lemat}\cite{tan}
\label{nullicycle}
If $n\equiv{0}\mbox{ }(mod\mbox{ }4)$, then $\eta(C_n)=2$. Otherwise, $\eta(C_n)=0$.
\end{lemat}

The following result computes the nullity of a unicyclic graph from the nullity of its pendant trees.

\begin{lemat}
\cite{nulidade}
\label{t6}
Let $G$ be a unicyclic graph and let $C$ be its cycle. If $G$ is of Type $I$ and $v\in{V(C)}$ be matched in $G\lbrace{v}\rbrace$, then $$\eta(G)=\eta(G\lbrace{v}\rbrace)+\eta(G-G\lbrace{v}\rbrace).$$
If $G$ is of Type $II$, then $$\eta(G)=\eta(G-C)+\eta(C).$$
\end{lemat}

\section{Null space of unicyclic graphs of Type I}
\label{sec:nulltipoI}

In this section, we obtain a basis for the null space of a unicyclic graph $G$ of Type $I$ using a basis for  $\mathcal{N}(G\{v\})$ and $\mathcal{N}(G-G\{v\})$.

\begin{defn}
Let $H$ be a subgraph of a graph $G$. Each vector $x\upharpoonright^{G}_{H}$ of the extended null space of $H$, denoted by $\mathcal{N}(H)\upharpoonright^{G}_{H}$, is constructed by extending a vector $x\in{\mathcal{N}(H)}$ as follows:
\begin{itemize}
\item[\textbf{1-}] For any $v\in{V(G)-V(H)}$, ${\left(x\upharpoonright^{G}_{H}\right)}_v=0$;
\item[\textbf{2-}] For any $v\in{V(H)}$, ${\left(x\upharpoonright^{G}_{H}\right)}_v=x_v$; .
\end{itemize}

\end{defn}

In the next results, we obtain a basis for the null space of unicyclic graphs of Type $I$.

\begin{prop}\label{t21}
 If $G$ is a unicyclic graph of Type $I$ and $G\lbrace{v_i}\rbrace$ is a pendant tree such that $v_i\notin{Supp(G\lbrace{v_i}\rbrace)}$, then $$\mathcal{N}(G\lbrace{v_i}\rbrace)\upharpoonright^{G}_{G\lbrace{v_i}\rbrace}\subseteq{\mathcal{N}(G)}.$$
\end{prop}

\begin{proof}
Consider the unicyclic graph $G$ of Type $I$ with cycle $C=\{v_{1}v_{2}\cdots v_{k}v_{1}\}$ and its corresponding pendant trees. If we order the columns of the adjacency matrix of $G$ from $V(G\lbrace{v_1}\rbrace)$ to $V(G\lbrace{v_k}\rbrace)$, then $A(G)$ can have the following format

\begin{scriptsize}
\begin{eqnarray*}
A(G)&=&\begin{blockarray}{ccccccccc}
         & V(G\lbrace{v_1}\rbrace)    &\ldots  & V(G\lbrace{v_{i-1}}\rbrace) & V(G\lbrace{v_i}\rbrace)    & V(G\lbrace{v_{i+1}}\rbrace)  &  \ldots &  V(G\lbrace{v_k}\rbrace) \\
\begin{block}{r[cccccccc]}
V(G\lbrace{v_1}\rbrace)      & A(G\lbrace{v_1}\rbrace)       & \cdots   &   \mathbf{0}                          &  \mathbf{0}         &  \mathbf{0}             &   \cdots      &  C         \\
\vdots                    &   \vdots     &   \ddots     &      \vdots                 &  \vdots  & \vdots       &  \cdots       &  \vdots   \\
V(G\lbrace{v_{i-1}}\rbrace)  &  \mathbf{0}^t      & \cdots & A(G\lbrace{v_{i-1}}\rbrace) &   M      &   \mathbf{0} &  \cdots &    \mathbf{0}        \\
V(G\lbrace{v_i}\rbrace)      &  \mathbf{0}^t      & \cdots &  M^t     &  A(G\lbrace{v_{i}}\rbrace)    & B        &    \cdots &      \mathbf{0}        \\
V(G\lbrace{v_{i+1}}\rbrace)  & \mathbf{0}^t       & \cdots &  \mathbf{0}^t     &  B^t        &  A(G\lbrace{v_{i+1}}\rbrace)  &  \cdots       &  \mathbf{0}  \\
\vdots   &    \vdots           & \cdots & \vdots &  \vdots        &  \vdots       & \ddots  &        \vdots            \\
V(G\lbrace{v_k}\rbrace) &  C^t     & \cdots &  \mathbf{0}^t     &  \mathbf{0}^t        &  \mathbf{0}^t       &  \cdots       &  A(G\lbrace{v_k}\rbrace)                  \\
\end{block}
\end{blockarray}.\\
\end{eqnarray*}
\end{scriptsize}

Where $M$, $B$ and $C$ are submatrices with almost all null entries, except for the entries corresponding to the adjacencies between  $v_{i-1}$ and $v_i$, $v_{i}$ and $v_{i+1}$, $v_{1}$ and $v_k$, respectively.

Consider $x\in{\mathcal{N}(G\lbrace{v_i}\rbrace)}$ and observe that as $v_i\notin{Supp(G\lbrace{v_i}\rbrace)}$ we have that  $x_{v_i}={0}$.
Thus
\begin{eqnarray*}
A(G)\cdot\left(x\upharpoonright^{G}_{G\lbrace{v_i}\rbrace}\right)=\left[\begin{array}{cccccccc}
   \mathbf{0} \\
    Mx  \\
    A(G\lbrace{v_{i}}\rbrace)x\\
      B^t.x\\
      \mathbf{0}\\
\end{array}\right]=\left[\begin{array}{cccccccc}
             \mathbf{0}\\
    x_{v_i}.1  \\
    A(G\lbrace{v_{i}}\rbrace)x\\
      x_{v_i}.1\\
      \mathbf{0}\\
\end{array}\right]=\mathbf{0}.
\end{eqnarray*}
Therefore, the extended vector $x\upharpoonright^{G}_{G\lbrace{v_i}\rbrace}\in \mathcal{N}(G)$ and it implies that $\mathcal{N}(G\lbrace{v_i}\rbrace)\upharpoonright^{G}_{G\lbrace{v_i}\rbrace}\subseteq{\mathcal{N}(G)}$.
\end{proof}

\begin{prop}
\label{teo1}
Let $G$ be a unicyclic graph and $G\lbrace{v}\rbrace$ a pendant tree.
\hspace{0.2cm}Let $u,w\in{N(v)\cap{V(G-G\lbrace{v}\rbrace)}}$.
If $x\in\mathcal{N}(G-G\lbrace{v}\rbrace)$ such that $x_u+x_w=0$, then $$x\upharpoonright^{G}_{G-G\lbrace{v}\rbrace}\in{\mathcal{N}(G)}.$$
\end{prop}
\begin{proof}
We observe that the condition $u,w\in{N(v)\cap{V(G-G\lbrace{v}\rbrace)}}$ means that the  vertices $u$ and $w$ are in the cycle $C$ of $G$ and are adjacent to $v$. Thus, we can sort the adjacency matrix of $G$ as follows
\begin{eqnarray*}
A(G)&=&\begin{blockarray}{cccc}
         & V(G-G\lbrace{v}\rbrace)    &V(G\lbrace{v}\rbrace)  \\
\begin{block}{r[ccc]}
 V(G-G\lbrace{v}\rbrace) & A(G-G\lbrace{v}\rbrace)  &            M    \\
V(G\lbrace{v}\rbrace)   &   M^{t}                  &   A(G\lbrace{v}\rbrace)    \\
\end{blockarray}.\\
\end{eqnarray*}
\noindent Where $M$ is a submatrix with almost all null entries, except for the entries corresponding to the adjacencies between $v$ and $u$ and  $v$ and $w$.
Given $x\in\mathcal{N}(G-G\lbrace{v}\rbrace)$ we construct the extended vector $x\upharpoonright^{G}_{G-G\lbrace{v}\rbrace}$. Using the hypothesis that $x_u+x_w=0$ we obtain that
\begin{eqnarray*}
A(G)x\upharpoonright^{G}_{G-G\lbrace{v}\rbrace}=\left[\begin{array}{cccccccc}
   A(G-G\lbrace{v}\rbrace)x+M.\mathbf{0}  \\
    M^{t}x+A(G\lbrace{v}\rbrace).\mathbf{0}\\
\end{array}\right]=\left[\begin{array}{cccccccc}
     \mathbf{0}\\
    x_{u}+x_{w}\\
\end{array}\right]=\mathbf{0}.
\end{eqnarray*}
Hence, $x\upharpoonright^{G}_{G-G\lbrace{v}\rbrace}\in{\mathcal{N}(G)}$.
\end{proof}

Now, we are ready to prove the following theorem.

\begin{teore}\label{teo11}
Let $G$ be a unicyclic graph of Type $I$, $G\lbrace{v}\rbrace$ its pendant tree such that $v\notin{Supp(G\lbrace{v}\rbrace)}$ and $u,w\in{N(v)\cap{V(G-G\lbrace{v}\rbrace)}}$. If for all ${x}\in{\mathcal{N}(G-G\{v\})}$ we have $x_u+x_w=0$, then  $$\mathcal{N}(G\lbrace{v}\rbrace)\upharpoonright^{G}_{G\lbrace{v}\rbrace}\bigoplus\mathcal{N}(G-G\lbrace{v}\rbrace)\upharpoonright^{G}_{G-G\lbrace{v}\rbrace}={\mathcal{N}(G)}.$$
\end{teore}
\begin{proof}
Using the fact that $\mathcal{N}(G\lbrace{v}\rbrace)\upharpoonright^{G}_{G\lbrace{v}\rbrace}\bigcap\mathcal{N}(G-G\lbrace{v}\rbrace)\upharpoonright^{G}_{G-G\lbrace{v}\rbrace}=\lbrace{\mathbf{0}}\rbrace$ and propositions \ref{t21} and \ref{teo1} we have that $\mathcal{N}(G\lbrace{v}\rbrace)\upharpoonright^{G}_{G\lbrace{v}\rbrace}\bigoplus\mathcal{N}(G-G\lbrace{v}\rbrace)\upharpoonright^{G}_{G-G\lbrace{v}\rbrace}\subseteq\mathcal{N}(G)$. Moreover, by Lemma \ref{t6}  $\mathcal{N}(G)$ and  $\mathcal{N}(G\lbrace{v}\rbrace)\upharpoonright^{G}_{G\lbrace{v}\rbrace}\bigoplus\mathcal{N}(G-G\lbrace{v}\rbrace)\upharpoonright^{G}_{G-G\lbrace{v}\rbrace}$ have the same dimension. Therefore, $$\mathcal{N}(G\lbrace{v}\rbrace)\upharpoonright^{G}_{G\lbrace{v}\rbrace}\bigoplus\mathcal{N}(G-G\lbrace{v}\rbrace)\upharpoonright^{G}_{G-G\lbrace{v}\rbrace}={\mathcal{N}(G)}.$$
\end{proof}

The next Lemma ensures that the set $\lbrace{t_1,t_2,\ldots,t_k}\rbrace$ in Proposition \ref{prob1} is non-empty.

\begin{lemat}\label{lematec2}
Let $T$ be a tree. If $v\notin{Supp(T)}$, then $N(v)\cap{Supp(T-v)}\neq{\emptyset}.$
\end{lemat}

\begin{proof}
Let $M\in\mathcal{M}(T)$ and $w\in{N(v)}$ such that $\lbrace{v,w}\rbrace\in{M}$.
As $v\notin{Supp(T)}$ then $v$ is saturated by $M$ according to Lemma \ref{t11}.
 We will prove that $M-\lbrace\lbrace{v,w}\rbrace\rbrace\in\mathcal{M}(T-v)$. Otherwise, there would be a matching $M^{'}$ in $T-v$ such that $\vert{M^{'}}\vert>\vert{M-\lbrace\lbrace{v,w}\rbrace\rbrace}\vert=\vert{M}\vert-1$. Observe that $M^{'}$ is a matching in $T$ as well. As $M\in\mathcal{M}(T)$ we have that $\vert{M}\vert\geq\vert{M^{'}}\vert$. Thus
\begin{eqnarray}\label{equa6}
\vert{M}\vert-1<\vert{M^{'}}\vert\leq{\vert{M}\vert}.
\end{eqnarray}
The inequality  \eqref{equa6} implies $\vert{M^{'}}\vert={\vert{M}\vert}$ because $\vert{M}\vert-1$ and $\vert{M}\vert$ are integer numbers. And, it means that $M^{'}\in\mathcal{M}(T)$. Which is a contradiction, because $M^{'}$ does not sature $v$ and all maximum matchings in $T$ sature $v$. Hence, $M-\lbrace\lbrace{v,w}\rbrace\rbrace\in\mathcal{M}(T-v)$. Moreover, $w$ is non-saturated by the maximum matching $M-\lbrace\lbrace{v,w}\rbrace\rbrace$ in $T-v$, and it means that $w\in{Supp(T-v)}$ by Lemma \ref{t11}.
\end{proof}

\begin{prop}\label{prob1}
Let $G$ be a unicyclic graph of Type $I$, $G\lbrace{v}\rbrace$ its pendant tree such that $v\notin{Supp(G\lbrace{v}\rbrace)}$ and $u,w\in{N(v)\cap{V(G-G\lbrace{v}\rbrace)}}$. Let $N(v)\cap{Supp(G\lbrace{v}\rbrace-v)}=\lbrace{t_1,t_2,\ldots,t_k}\rbrace$. If $y\in\mathcal{N}(G\lbrace{v}\rbrace-v)$ and there is a $x\in{\mathcal{N}(G-G\lbrace{v}\rbrace)}$ such that $x_u+x_w\neq{0}$, then $s\in\mathcal{N}(G)$, where $s=c\cdot x\upharpoonright^{G}_{G-G\lbrace{v}\rbrace}+y\upharpoonright^{G}_{G\lbrace{v}\rbrace-v}$ with $c=\frac{-\sum\limits_{i=1}^{k}y_{t_i}}{x_u+x_w}$.
\end{prop}
\begin{proof}
We can sort the adjacency matrix of $G$ as
\begin{eqnarray*}
A(G)&=&\begin{blockarray}{ccccc}
                        & V(G-G\lbrace{v}\rbrace) & v  & V(G\lbrace{v}\rbrace-v)  \\
\begin{block}{r[cccc]}
V(G-G\lbrace{v}\rbrace) & A(G-G\lbrace{v}\rbrace)  &    \mathbf{a}   &     \mathbf{0}                        \\
  v                     &    \mathbf{a}^{t}                 &    0   &     \mathbf{b}^t                      \\
V(G\lbrace{v}\rbrace-v) &   \mathbf{0}^{t}                  &    \mathbf{b}    &   A(G\lbrace{v}\rbrace-v)    \\
\end{blockarray}.\\
\end{eqnarray*}

\noindent Where $\mathbf{a}$ is a submatrix with almost all null entries, except for the entries corresponding to the adjacencies between $v$ and $u$ and $v$ and $w$. The submatrix $\mathbf{b}$ also has almost all null entries, except for the entries corresponding to the adjacencies between $v$ and the vertices of  $G\lbrace{v}\rbrace-v$. Then

\begin{eqnarray*}
A(G)\cdot s&=&\left[\begin{array}{cccccccc}
   \left(\frac{-\sum\limits_{i=1}^{k}y_{t_i}}{x_u+x_w}\right)A(G-G\lbrace{v}\rbrace)x \\
      \left(\frac{-\sum\limits_{i=1}^{k}y_{t_i}}{x_u+x_w}\right)\mathbf{a}^{t}x+\mathbf{b}^{t}y \\
      A(G\lbrace{v}\rbrace-v)y
\end{array}\right]=\left[\begin{array}{cccccccc}
        \mathbf{0}\\
      \left(\frac{-\sum\limits_{i=1}^{k}y_{t_i}}{x_u+x_w}\right)(x_u+x_w)+\sum\limits_{i=1}^{k}y_{t_i}\\
      \mathbf{0}
\end{array}\right]=\mathbf{0}.
\end{eqnarray*}

Therefore, $s\in\mathcal{N}(G).$
\end{proof}

We need the next lemma from \cite{peter} to ensure the existence of the vector $y$ in Theorem \ref{teobasistipoi}.

\begin{lemat}
\label{suppcompute2}
Let $G$ be a graph with $n$ vertices and $S$ a subset of $\mathbb{R}^{n}$ with dimension $k>0$. Then there exists a basis $\mathcal{B}=\lbrace{b_1,\ldots,b_k}\rbrace$ of $S$ satisfying $Supp(G)=Supp_{G}(b_i)$ for all $i=1,2,\ldots,k.$
\end{lemat}

The next result gives a basis for $\mathcal{N}(G)$, where $G$ is a unicyclic graph of Type $I$, using basis of $\mathcal{N}(G\{v\})$ and $\mathcal{N}(G-G\{v\})$.

\begin{teore}\label{teobasistipoi}
Let $G$ be a unicyclic graph of Type $I$, $G\lbrace{v}\rbrace$ its pendant tree such that $v\notin{Supp(G\lbrace{v}\rbrace)}$ and $u,w\in{N(v)\cap{V(G-G\lbrace{v}\rbrace)}}$. Let $N(v)\cap{Supp(G\lbrace{v}\rbrace-v)}=\lbrace{t_1,t_2,\ldots,t_k}\rbrace$. Let $\lbrace{w_1,w_2,\ldots,w_t}\rbrace$ be a basis of $\mathcal{N}(G\lbrace{v}\rbrace)$. Let $\lbrace{b_1,b_2,\ldots,b_p}\rbrace\cup\lbrace{a_1,a_2,\ldots,a_r}\rbrace$ be a basis of $\mathcal{N}(G-G\lbrace{v}\rbrace)$ such that $({b_i})_u+({b_i})_w\neq{0}$ and $({a_j})_u+({a_j})_w{=}{0}$, for $i=1,2,\ldots,p$ and $j=1,2,\ldots,r$. Let $y\in\mathcal{N}(G\lbrace{v}\rbrace-v)$ such that $Supp(G\lbrace{v}\rbrace-v)=Supp_{G\lbrace{v}\rbrace-v}(y)$. Let $c_i=\frac{-\sum\limits_{s=1}^{k}y_{t_s}}{({b_i})_u+({b_i})_w}$, for $i=1,2,\ldots,p$. Define
$$\mathcal{B}_1=\bigcup\limits_{i=1}^{p}\left\lbrace{\left[\begin{array}{cccccccc}
        c_ib_i\\
         0\\
         y\\
\end{array}\right]}\right\rbrace,\mbox{ }
\mathcal{B}_2=\bigcup\limits_{i=1}^{r}\left\lbrace{a_i\upharpoonright^{G}_{G-G\lbrace{v}\rbrace}}\right\rbrace\mbox{ and }
\mathcal{B}_3=\bigcup\limits_{i=1}^{t}\left\lbrace{w_i\upharpoonright^{G}_{G\lbrace{v}\rbrace}}\right\rbrace.$$

\noindent Then $\mathcal{B}_1\cup\mathcal{B}_2\cup\mathcal{B}_3$ is a basis of $\mathcal{N}(G)$.
\end{teore}

\begin{proof}
Notice that $\mathcal{B}_1\cup\mathcal{B}_2\cup\mathcal{B}_3\subseteq{\mathcal{N}(G)}$ by propositions \ref{prob1}, \ref{teo1} and \ref{t21}, respectively. Moreover,  $\vert\mathcal{B}_1\cup\mathcal{B}_2\cup\mathcal{B}_3\vert=\eta(G\lbrace{v}\rbrace)+\eta(G-G\lbrace{v}\rbrace)$. Now, we will prove that $\mathcal{B}_1\cup\mathcal{B}_2\cup\mathcal{B}_3$ is a set of linearly independent vectors. Given $\alpha_1,\alpha_2,\ldots,\alpha_p,\beta_1,\beta_2,\ldots,\beta_r,\gamma_1,\gamma_2,\ldots,\gamma_t\in\mathbb{R}$ and consider the following linear combination

\begin{eqnarray}\label{equa1}
\sum\limits_{i=1}^{p}\left(\alpha_i\left[\begin{array}{cccccccc}
        c_ib_i\\
         0\\
         y\\
\end{array}\right]\right)+\sum\limits_{j=1}^{r}\left(\beta_j\left[\begin{array}{cccccccc}
        a_j\\
       \mathbf{0}\\
       \end{array}\right]\right)+\sum\limits_{\ell=1}^{t}\left(\gamma_{\ell}\left[\begin{array}{cccccccc}
       \mathbf{0}\\
       w_{\ell}\\
       \end{array}\right]\right)&=&\mathbf{0}\nonumber\\
      \left[\begin{array}{cccccccc}
        \sum\limits_{i=1}^{p}\alpha_i{c_ib_i}+\sum\limits_{j=1}^{r}\beta_j{a_j}\\
         \sum\limits_{i=1}^{p}\alpha_i{\left[\begin{array}{cccccccc}
         0\\
         y\\
\end{array}\right]}+\sum\limits_{\ell=1}^{t}\gamma_{\ell}w_{\ell}\\
\end{array}\right]&=&\mathbf{0}.
\end{eqnarray}

Then $\sum\limits_{i=1}^{p}\alpha_i{c_ib_i}+\sum\limits_{j=1}^{r}\beta_j{a_j}=0$. But  $\lbrace{b_1,b_2,\ldots,b_p}\rbrace\cup\lbrace{a_1,a_2,\ldots,a_r}\rbrace$ is a basis of $\mathcal{N}(G-G\lbrace{v}\rbrace)$ and $c_i\neq{0}$, hence $\alpha_i=\beta_j=0$ for $i=1,2,\ldots,p$ and $j=1,2,\ldots,r$. Using \eqref{equa1} we have that $\sum\limits_{\ell=1}^{t}\gamma_{\ell}w_{\ell}=0$. As  $\lbrace{w_1,w_2,\ldots,w_t}\rbrace$ is a basis of $\mathcal{N}(G\lbrace{v}\rbrace)$, then $\gamma_{\ell}=0$ for $\ell=1,2,\ldots,t$.

We have proved that $\mathcal{B}_1\cup\mathcal{B}_2\cup\mathcal{B}_3$ is a set of linearly independent vectors in $\mathcal{N}(G)$ and by Lemma \ref{t6} $\eta(G)=\vert\mathcal{B}_1\cup\mathcal{B}_2\cup\mathcal{B}_3\vert$. Therefore $\mathcal{B}_1\cup\mathcal{B}_2\cup\mathcal{B}_3$ is a basis of $\mathcal{N}(G)$.
\end{proof}

\section{Null space of unicyclic graphs of Type II}
\label{sec:nulltipoII}
In this section, we obtain a basis for the null space of a unicyclic graph $G$ of Type $II$ using basis for  $\mathcal{N}(G-C)$ and $\mathcal{N}(G\{v_{i}\})$. Where $C$ is the cycle of $G$ and $v_{i}$  is a vertex in $C$.

The next lemma will be used in Proposition \ref{tipe2subspace}.

\begin{lemat}
\cite{Maikon}
\label{non supp}
Let $G$ be a unicyclic graph and $C$ its cycle. Let $G\lbrace{v}\rbrace$ be a pendant tree such that $v\in{Supp(G\lbrace{v}\rbrace)}$. If $u\in{N(v)\cap{V(G\lbrace{v}\rbrace)}}$, then $u\notin{Supp(G-C)}$.
\end{lemat}

In the following proposition we find vectors of the $\mathcal{N}(G)$ using the null space of the subtrees obtained from $G-C$.

\begin{prop}
\label{tipe2subspace}
If $G$ is a unicyclic graph of Type $II$ and $C$ its cycle, then $\mathcal{N}(G-C)\upharpoonright^{G}_{G-C}\subseteq{\mathcal{N}(G)}$.
\end{prop}
\begin{proof}
Consider $G-C=\bigcup\limits^{k}_{i=1}{T_i}$, where $T_i$ is a connected component of $G-C$.
\noindent Given $x\in\mathcal{N}(T_i)$. We will prove that $x\upharpoonright^{G}_{T_i}\in{\mathcal{N}(G)}$.
Remember that given $v\in V(C)$, by Lemma \ref{c3}, we have that $v\in Supp(G\lbrace{v}\rbrace)$. If  $u_{i}\in V(T_{i})\cap N(v)$, then by Lemma \ref{non supp} $u_i\notin Supp(G-C)$. It means that $u_i\notin{Supp(T_{i})}$, that is, $x_{u_i}=0$.

Note that, we can order the adjacency matrix of $G$ as follows

\begin{scriptsize}
\begin{eqnarray*}
A(G)&=&\begin{blockarray}{ccccccccc}
         & V(T_1)    & V(T_2)  & \cdots & V(T_i) & \cdots  & V(T_k) &  V(C) \\
\begin{block}{r[cccccccc]}
V(T_1)      & A(T_1)       &  \mathbf{0}   &   \cdots    &  \mathbf{0}         &  \cdots           &  \mathbf{0}  &  M_1        \\
V(T_2) &  \mathbf{0}^t      &   A(T_2)    &      \vdots                 &  \mathbf{0}  & \vdots        &   \mathbf{0}   &   M_2  \\
\vdots &  \vdots     & \vdots & \ddots &   \vdots    &   \vdots &  \vdots &   \vdots       \\
V(T_i)      &  \mathbf{0}^t      & \mathbf{0}^t &  \cdots    &  A(T_i)    & \cdots       &    \mathbf{0} &      M_i        \\
\vdots  & \vdots       & \vdots &  \vdots      &  \vdots      & \ddots  &  \vdots       &  \vdots  \\
V(T_k)   &   \mathbf{0}^t          & \mathbf{0}^t & \cdots &  \mathbf{0}^t  &  \cdots  & A(T_k) &        M_k           \\
V(C) &  M_1^t     & M_2^t &  \cdots     &  M_i^t       &  \cdots       &  M_k^t       &  A(C)                  \\
\end{block}
\end{blockarray}.\\
\end{eqnarray*}
\end{scriptsize}
	
\noindent Where $M_i$ is a submatrix of $A(G)$ with almost all null entries, with the exception of the single entry corresponding to the adjacency between vertices $v\in{V(C)}$ and $u_i\in{V(T_i)}$, for $i\in{\lbrace{1,2,\ldots,k}\rbrace}$. Thus
\begin{eqnarray*}
A(G)\cdot x\upharpoonright^{G}_{T_i}=\left[\begin{array}{cccccccc}
    \mathbf{0} \\
    A(T_i)x\\
      M_i^t.x\\
      \mathbf{0} \\
\end{array}\right]=\left[\begin{array}{cccccccc}
            \mathbf{0}\\
    \mathbf{0}\\
      x_{u_i}.1\\
      \mathbf{0}\\
\end{array}\right]=\mathbf{0}.
\end{eqnarray*}

Therefore, $\mathcal{N}(T_i)\upharpoonright^{G}_{T_i}\subseteq\mathcal{N}(G)$.
And, since that $G-C$ is the union of disjoint graphs we have that $$\bigoplus\limits_{i=1}^{k}\mathcal{N}(T_i)\upharpoonright^{G}_{T_i}=\mathcal{N}(G-C)\upharpoonright^{G}_{G-C}\subseteq{\mathcal{N}(G)}.$$
\end{proof}

The next theorem shows that the null space of unicyclic graphs of Type $II$ with a cycle that has lenght diferent of $4k$ is equal to the extended null space of the forest $G-C$.

\begin{teore}\label{t40}
 If $G$ is a unicyclic graph of Type $II$ and $C$ its cycle such that $\vert{V(C)}\vert\neq{4}k$, with $k\in\mathbb{N}$, then $\mathcal{N}(G-C)\upharpoonright^{G}_{G-C}={\mathcal{N}(G)}$.
\end{teore}

\begin{proof}
As $\vert{V(C)}\vert\neq{4}k$, we have that $\eta(C)=0$ by Lemma \ref{nullicycle}. Using this in Lemma \ref{t6} we obtain that $\eta{(G-C)}=\eta{(G)}$. And, the Proposition \ref{tipe2subspace} gives that $\mathcal{N}(G-C)\upharpoonright^{G}_{G-C}\subseteq{\mathcal{N}(G)}$. These facts imply that $\mathcal{N}(G-C)\upharpoonright^{G}_{G-C}={\mathcal{N}(G)}$.
\end{proof}

According to lemmas \ref{nullicycle} and \ref{t6}, we still have to find two vectors for $\mathcal{N}(G)$ when $\vert{V(C)}\vert=4k$. And, in the next proposition we construct these two vectors.

\begin{prop}\label{pc4}
 Let $G$ be a unicyclic graph of Type $II$ and $C$ its cycle. Let $V(C)=\lbrace{v_1,v_2,\ldots,v_{4k}}\rbrace$, for some $k\in\mathbb{N}$, such that $N(v_i)\cap{V(C)}=\lbrace{v_{i-1},v_{i+1}}\rbrace$ and $v_0=v_{4k}$. Let $x(v_i)\in\mathcal{N}(G\lbrace{v_i}\rbrace)$ such that $Supp(G\lbrace{v_i}\rbrace)=Supp_{G\lbrace{v_i}\rbrace}(x(v_i))$. If $x(v_i)_{v_i}=1$, then  $z_1,z_2\in\mathcal{N}(G)$, where
 $$z_1=\sum\limits_{i=1}^{2k}(-1)^ix(v_{2i-1})\upharpoonright^{G}_{G\lbrace{v_{2i-1}}\rbrace}\mbox{ and }z_2=\sum\limits_{i=1}^{2k}(-1)^ix(v_{2i})\upharpoonright^{G}_{G\lbrace{v_{2i}}\rbrace}.$$

\end{prop}

\begin{proof}
Notice that
\begin{eqnarray*}
\left(A(G)x(v_i)\upharpoonright^{G}_{G\lbrace{v_{i}}\rbrace}\right)_{v_j}&=&\sum\limits_{u\in{N(v_j)}}\left(x(v_i)\upharpoonright^{G}_{G\lbrace{v_{i}}\rbrace}\right)_u\\
&=&\begin{cases}
1, & \mbox{if }j\in\lbrace{1,\ldots,4k}\rbrace\mbox{ and }j=i-1\mbox{ or }j=i+1;\\
0, & \mbox{otherwise}.
\end{cases}
\end{eqnarray*}

\noindent Thus
\begin{eqnarray*}
\left({A(G)z_1}\right)_{v_j}&=&\left(A(G)\sum\limits_{i=1}^{2k}(-1)^ix(v_{2i-1})\upharpoonright^{G}_{G\lbrace{v_{2i-1}}\rbrace}\right)_{v_j}\\
&=&\sum\limits_{i=1}^{2k}(-1)^i\left(A(G)x(v_{2i-1})\upharpoonright^{G}_{G\lbrace{v_{2i-1}}\rbrace}\right)_{v_j}\\
&=&\sum\limits_{i=1}^{2k}(-1)^i\sum\limits_{u\in{N(v_j)}}\left(x(v_{2i-1})\upharpoonright^{G}_{G\lbrace{v_{2i-1}}\rbrace}\right)_u\\
&=&\begin{cases}
\sum\limits_{i=1}^{2k}(-1)^i, & \mbox{if }j\in\lbrace{1,\ldots,4k}\rbrace\mbox{ and }j=2i-2\mbox{ or }j=2i;\\
0, & \mbox{otherwise}.
\end{cases}\\
&=&0.
\end{eqnarray*}

\noindent And similarly  we conclude that $A(G)z_2=\mathbf{0}.$

\noindent Therefore, $z_1$ and $z_2\in\mathcal{N}(G)$.
\end{proof}

Now, we are able to present a basis of the null space of a unicyclic graph of Type $II$ when $\vert{V(C)}\vert=4k$.

\begin{teore}\label{tbasisc4}
 Let $G$ be a unicyclic graph of Type $II$ and $C$ its cycle. Let $V(C)=\lbrace{v_1,v_2,\ldots,v_{4k}}\rbrace$, for some $k\in\mathbb{N}$, such that $N(v_i)\cap{V(C)}=\lbrace{v_{i-1},v_{i+1}}\rbrace$ and $v_0=v_{4k}$. Let $x(v_i)\in\mathcal{N}(G\lbrace{v_i}\rbrace)$ such that  $Supp(G\lbrace{v_i}\rbrace)=Supp_{G\lbrace{v_i}\rbrace}(x(v_i))$. Let $\lbrace{w_1,w_2,\ldots,w_t}\rbrace$ be a basis of $\mathcal{N}(G-C)$. If $x(v_i)_{v_i}=1$, then $\mathcal{B}\cup{\lbrace{z_1,z_2}\rbrace}$ is a basis of $\mathcal{N}(G)$, where  $$ \mathcal{B}=\bigcup\limits_{i=1}^{t}\left\lbrace{w_i\upharpoonright^{G}_{G-C}}\right\rbrace,\mbox{ }z_1=\sum\limits_{i=1}^{2k}(-1)^ix(v_{2i-1})\upharpoonright^{G}_{G\lbrace{v_{2i-1}}\rbrace}\mbox{ and } z_2=\sum\limits_{i=1}^{2k}(-1)^ix(v_{2i})\upharpoonright^{G}_{G\lbrace{v_{2i}}\rbrace}.$$
\end{teore}

\begin{proof}
Using the propositions \ref{tipe2subspace} and \ref{pc4} we observe that $\mathcal{B}\cup{\lbrace{z_1,z_2}\rbrace}\subseteq{\mathcal{N}(G)}$. Using the Lemma \ref{nullicycle} and the hypothesis that $\vert{V(C)}\vert=4k$ we have that $\eta(C)=2$.
Moreover,  $\vert\mathcal{B}\cup{\lbrace{z_1,z_2}\rbrace}\vert=t+2=\eta(G-C)+\eta(C)$. And, as $G$ is a unicyclic graph of Type $II$, we have $\eta(G)=\eta(G-C)+\eta(C)$ according to Lemma \ref{t6}. Therefore, $\vert\mathcal{B}\cup{\lbrace{z_1,z_2}\rbrace}\vert=\eta(G)$ and now we just have to prove that
$\mathcal{B}\cup{\lbrace{z_1,z_2}\rbrace}$ is a linearly independent set of vectors.
Given $\alpha_1,\alpha_2,\ldots,\alpha_t,\gamma_1,\gamma_2\in\mathbb{R}$, suppose that

\begin{eqnarray}\label{equa9}
\gamma_1z_1+\gamma_2z_2+\sum\limits_{i=1}^{t}\alpha_iw_i\upharpoonright^{G}_{G-C}=0.
\end{eqnarray}

For $j\in\lbrace{1,\ldots,4k}\rbrace$, we have that
\begin{eqnarray}\label{equa8}
\left(\gamma_1z_1+\gamma_2z_2+\sum\limits_{i=1}^{t}\alpha_iw_i\upharpoonright^{G}_{G-C}\right)_{v_j}=\begin{cases}
\gamma_1(-1)^i, & \mbox{if }j=2i-1;\\
\gamma_2(-1)^i, & \mbox{if }j=2i.
\end{cases}
\end{eqnarray}

Comparing \eqref{equa9} and \eqref{equa8} we can conclude that $\gamma_1(-1)^i =\gamma_2(-1)^i=0$, and it means that $\gamma_1=\gamma_2=0$. Substituting the values of $\gamma_1$ and $\gamma_2$ in \eqref{equa9}, we obtain the following

\begin{eqnarray}\label{equa12}
\sum\limits_{i=1}^{t}\alpha_iw_i\upharpoonright^{G}_{G-C}=0=\left[\begin{array}{cccccccc}
       \sum\limits_{i=1}^{t}\alpha_iw_{i}\\
     0\\
       \end{array}\right].
\end{eqnarray}

Then $\sum\limits_{i=1}^{t}\alpha_i w_{i}=0$, and using the fact that $\lbrace{w_1,w_2,\ldots,w_t}\rbrace$ is a basis of $\mathcal{N}(G-C)$ we conclude that $\alpha_i = 0$ for $i=1,2,\ldots,t$. And it proves that $\mathcal {B}\cup\lbrace {z_1, z_2}\rbrace $ is a linearly independent set of vectors.
\end{proof}
\section{Support, Core and $N$-vertices sets of Unicyclic Graphs of Type $I$}\label{sec:suptipoI}

In this section, we present the null decomposition of graphs and we study the support, core and $N$-vertices of unicyclic graphs of Type $I$.
We divide the unicyclic graphs of Type I into four disjoint sets. More precisely, the unicyclic graphs of Type I must satisfy the hypothesis in one of the following propositions: \ref{ti1}, \ref{ti4}, \ref{ti6} or \ref{ti3}.

The definition of null decomposition of trees from \cite{tree} can directly be extended for graphs in general as follows.

\begin{defn}
Let $G$ be a graph. The $S$-graph of $G$, denoted by $\mathcal{G}_S(G)$, is the induced subgraph defined by the closed neighborhood of  $Supp(G)$ in $G$: $\mathcal{G}_S(G)=G\langle{N[Supp(G)]}\rangle.$
The $N$-graph  of $G$, denoted by $\mathcal{G}_N(G)$, is defined as the remaining graph as follows $\mathcal{G}_N(G)=G-\mathcal{G}_S(G).$
Then, the Null Decomposition of $G$ is the pair $\left(\mathcal{G}_S(G),\mathcal{G}_N(G)\right)$. And, $V(\mathcal{G}_N(G))$ is called the $N$-vertices set of $G$.
\end{defn}

\begin{defn}
The core of $G$, denoted by $Core(G)$, is defined to be the set of all the neighbors of the supported vertices of $G$:
$Core(G)=\bigcup\limits_{v\in{Supp(G)}}{N(v)}.$
\end{defn}

Now, we give an example of the null decomposition of a unicyclic graph.

\usetikzlibrary{shapes,snakes}
\tikzstyle{vertex}=[circle,draw,minimum size=0.1pt,inner sep=0.1pt]
\tikzstyle{edge} = [draw,thick,-]
\tikzstyle{matched edge} = [draw,snake=zigzag,line width=1pt,-]
\begin{figure}[h!]
\begin{center}
\begin{scriptsize}
\begin{center}
\begin{tikzpicture}[scale=1.1,auto,swap]
\node[draw,star,star points=9,star point ratio=0.6,label=below left:] (1) at (0,0) {$v_1$};
\node[draw,star,star points=9,star point ratio=0.6,label=below left:] (2) at (-1,-1) {$v_2$};
\node[draw,star,star points=9,star point ratio=0.6,label=below left:] (3) at (1,-1) {$v_3$};
\node[draw,star,star points=9,star point ratio=0.6,label=below left:] (4) at (0.5,-2) {$v_{4}$};
\node[draw,star,star points=9,star point ratio=0.6,label=below left:] (5) at (-0.5,-2) {$v_{5}$};
\node[draw,circle,label=below left:] (6) at (0,1) {$v_6$};
\node[draw,rectangle,label=below left:] (7) at (-1,1) {$v_7$};
\node[draw,rectangle,label=below left:] (8) at (1,1) {$v_8$};
\node[draw,rectangle,label=below left:] (9) at (0,2) {$v_9$};
\node[draw,star,star points=9,star point ratio=0.6,label=below left:] (10) at (1,0) {$v_{10}$};
\node[draw,star,star points=9,star point ratio=0.6,label=below left:] (11) at (2,-1) {$v_{11}$};
\node[draw,star,star points=9,star point ratio=0.6,label=below left:] (12) at (3,-1) {$v_{12}$};
\node[draw,star,star points=9,star point ratio=0.6,label=below left:] (13) at (4,-1) {$v_{13}$};
\node[draw,star,star points=9,star point ratio=0.6,label=below left:] (14) at (-2,-1) {$v_{14}$};
\node[draw,star,star points=9,star point ratio=0.6,label=below left:] (15) at (-3,-1) {$v_{15}$};
\node[draw,star,star points=9,star point ratio=0.6,label=below left:] (16) at (-4,-1) {$v_{16}$};
\node[draw,star,star points=9,star point ratio=0.6,label=below left:] (17) at (-3,0) {$v_{17}$};
\node[draw,star,star points=9,star point ratio=0.6star,star points=9,star point ratio=0.6,label=below left:] (18) at (-3,1) {$v_{18}$};
\node at (0,2.6) {$\mathcal{G}_S(G)$};
\node at (0,-2.7) {$\mathcal{G}_N(G)$};
\draw[dashed] (-0.3,2.3) -- (0.3,2.3);
\draw[dashed] (-0.3,2.3) -- (-0.3,1.3);
\draw[dashed] (0.3,2.3) -- (0.3,1.3);
\draw[dashed] (-0.3,1.3) --(-1.3,1.3);
\draw[dashed] (-1.3,1.3) --(-1.3,0.6);
\draw[dashed] (-1.3,0.6) -- (1.3,0.6);
\draw[dashed] (1.3,1.3) --(0.3,1.3);
\draw[dashed] (1.3,1.3) -- (1.3,0.6);
\draw[dashed] (-0.7,-2.4) --(0.7,-2.4);
\draw[dashed] (0.7,-2.4) --(1.2,-1.4);
\draw[dashed] (1.2,-1.4) --(4.4,-1.4);
\draw[dashed] (4.4,-1.4) --(4.4,-0.62);
\draw[dashed] (4.4,-0.62) --(1.5,-0.62);
\draw[dashed] (1.5,0.4) --(1.5,-0.62);
\draw[dashed] (-0.2,0.4) --(1.5,0.4);
\draw[dashed] (-0.7,-0.62) --(-0.2,0.4);
\draw[dashed] (-0.7,-0.62) --(-2.6,-0.62);
\draw[dashed] (-2.6,-0.62) --(-2.6,1.4);
\draw[dashed] (-2.6,1.4) --(-3.4,1.4);
\draw[dashed] (-3.4,1.4) --(-3.4,-0.62);
\draw[dashed] (-3.4,-0.62) --(-4.4,-0.62);
\draw[dashed] (-4.4,-0.62) --(-4.4,-1.4);
\draw[dashed] (-4.4,-1.4) --(-1.2,-1.4);
\draw[dashed] (-1.2,-1.4) --(-0.7,-2.4);
\foreach \from/\to in {1/2,1/3,3/4,4/5,5/2,1/6,6/9,6/7,6/8,1/10,3/11,12/11,12/13,17/15,14/15,2/14,15/16,17/18,5/4,3/11,12/13} {
 \draw (\from) -- (\to);}
 \foreach \source / \dest in {}
   \path[matched edge] (\source) -- (\dest);
\end{tikzpicture}
\end{center}
\end{scriptsize}
 \end{center}
 \begin{center}
\caption{{\small Null decomposition of the unicyclic graph $G$.}}\label{figura2}
\end{center}
\end{figure}
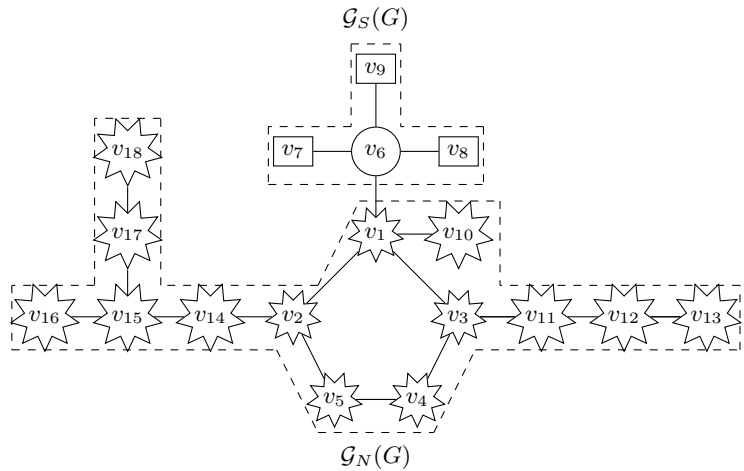
The unicyclic graph $G$ in Figure \ref{figura2} has $Supp(G)=\lbrace{v_7,v_8,v_9}\rbrace$ and $Core(G)=\lbrace{v_6}\rbrace$. Then, the $S$-graph of $G$ generated by the closed neighbourhood of the support consists of

\begin{eqnarray*}
\mathcal{G}_S(G)&=&G\langle{N[Supp(G)]}\rangle=G\langle{N[\lbrace{v_7,v_8,v_{9}}\rbrace]}\rangle=G\langle{\lbrace{v_6,v_7,v_8,v_{9}}\rbrace}\rangle.
\end{eqnarray*}
The $N$-graph of $G$ is given by
\begin{eqnarray*}
\mathcal{G}_N(G)=G-\mathcal{G}_S(G)=G\langle{\lbrace{v_1,v_2,v_3,v_4,v_5,v_{10},v_{11},v_{12},v_{13},v_{14},v_{15},v_{16},v_{17},v_{18}}\rbrace}\rangle.
\end{eqnarray*}
And, the $N$-vertices set of $G$ is $$V(\mathcal{G}_N(G))=\lbrace{v_1,v_2,v_3,v_4,v_5,v_{10},v_{11},v_{12},v_{13},v_{14},v_{15},v_{16},v_{17},v_{18}}\rbrace.$$

\begin{remark}
Let $G$ be a unicyclic graph of Type $I$, $G\lbrace{v}\rbrace$ its pendant tree such that $v\notin{Supp(G\lbrace{v}\rbrace)}$ and $u,w\in{N(v)\cap{V(G-G\lbrace{v}\rbrace)}}$. Note that $G$ necessarily satisfies the hypotheses of one of the propositions \ref{ti1}, \ref{ti4}, \ref{ti6} or \ref{ti3}. In fact, we can divide the set of unicyclic graphs of Type $ I$ into four disjoint classes according to the items below.

\begin{itemize}
\item[(1)]$\forall{x}\in\mathcal{N}(G-G\lbrace{v}\rbrace)$ we have that $x_u=x_w=0$, that is, $u,w\notin{Supp(G-G\lbrace{v}\rbrace)}.$ (Hypotheses of Proposition \ref{ti1})
\item[(2)] $\forall{x}\in\mathcal{N}(G-G\lbrace{v}\rbrace)$ we have that $x_u+x_w=0$, and there is $y\in\mathcal{N}(G-G\lbrace{v}\rbrace)$ such that $y_u\neq{0}$  and $v\in{Core(G\lbrace{v}\rbrace)}$ (Hypotheses of Proposition \ref{ti4})
\item[(3)]$\forall{x}\in\mathcal{N}(G-G\lbrace{v}\rbrace)$ we have that $x_u+x_w=0$, and there is $y\in\mathcal{N}(G-G\lbrace{v}\rbrace)$ such that  $y_u\neq{0}$ and $v\in{V(\mathcal{G}_N(G\lbrace{v}\rbrace))}$ (Hypotheses of Proposition \ref{ti6})
\item[(4)] There is $x\in\mathcal{N}(G-G\lbrace{v}\rbrace)\mbox{ such that }x_u+x_w\neq{0}$ (Hypotheses of Proposition \ref{ti3})
\end{itemize}
\end{remark}

In the next propositions we compute the support, core and $N$-vertices of unicyclic graphs of Type $I$ using its subtrees.

\begin{prop}\label{ti1}
Let $G$ be a unicyclic graph of Type $I$, $G\lbrace{v}\rbrace$ its pendant tree such that $v\notin{Supp(G\lbrace{v}\rbrace)}$ and $u,w\in{N(v)\cap{V(G-G\lbrace{v}\rbrace)}}$. If $u,w\notin{Supp(G-G\lbrace{v}\rbrace)}$, then

\begin{enumerate}[label=(\roman*)]
\item $Supp(G)=Supp(G\lbrace{v}\rbrace)\bigcup{Supp(G-G\lbrace{v}\rbrace)}$;
\item $Core(G)=Core(G\lbrace{v}\rbrace)\bigcup{Core(G-G\lbrace{v}\rbrace)}$;
\item $V(\mathcal{G}_N(G))=V(\mathcal{G}_N(G\lbrace{v}\rbrace))\bigcup{V(\mathcal{G}_N(G-G\lbrace{v}\rbrace)}$;
\item $Supp(G)$ is an independent set of $G$;
\item $Supp(G)\cap{Core(G)}=\emptyset.$
\end{enumerate}
\end{prop}
\begin{proof}
\begin{enumerate}[label=(\roman*)]
\item As $u,w\notin{Supp(G-G\lbrace{v}\rbrace)}$, then $x_u=x_w=0$ for all $x\in\mathcal{N}(G-G\lbrace{v}\rbrace)$. It means that  $x_u+x_w=0$. Analyzing the entries of the vectors of the basis of the $\mathcal{N}(G)$ in  Theorem \ref{teo11} we obtain the result.
\item Given a vertex $p\in{Core(G)}$, then there is a vertex $p_1\in{Supp(G)}$ such that $p\in{N(p_1)}$. Notice that, the only adjacencies between $G\lbrace{v}\rbrace$ and $G-G\lbrace{v}\rbrace$ occur between vertices $v$ and $u$ and $v$ and $w$, that is, ${V(G-G\lbrace{v}\rbrace)\cap{N(V(G\lbrace{v}\rbrace))}}=\lbrace{u,w}\rbrace$ and ${V(G\lbrace{v}\rbrace)}\cap{N(V(G-G\lbrace{v}\rbrace))}=\lbrace{v}\rbrace$. Thus, $N(V(G\lbrace{v}\rbrace)-\lbrace{v}\rbrace)\cap{V(G-G\lbrace{v}\rbrace)}=\emptyset$ and  $N(V(G-G\lbrace{v}\rbrace)-\lbrace{u,w}\rbrace)\cap{V(G\lbrace{v}\rbrace)}=\emptyset$. Since $v\notin{Supp(G\lbrace{v}\rbrace)}$ and $u,w\notin{Supp(G-G\lbrace{v}\rbrace)}$, then $Supp(G\lbrace{v}\rbrace)\subseteq{V(G\lbrace{v}\rbrace)-\lbrace{v}\rbrace}$ and  $Supp(G-G\lbrace{v}\rbrace)\subseteq{V(G-G\lbrace{v}\rbrace)-\lbrace{u,w}\rbrace}$. Therefore,  $N(Supp(G\lbrace{v}\rbrace))\cap{V(G-G\lbrace{v}\rbrace)}=\emptyset$ and $N(Supp(G-G\lbrace{v}\rbrace))\cap{V(G\lbrace{v}\rbrace)}=\emptyset$. 
Moreover, by item (i) we have that $Supp(G)=Supp(G\lbrace{v}\rbrace)\bigcup{Supp(G-G\lbrace{v}\rbrace)}$, thus $p_1\in{Supp(G\lbrace{v}\rbrace)}$ or $p_1\in{Supp(G-G\lbrace{v}\rbrace)}$. Hence, if $p_1\in{Supp(G\lbrace{v}\rbrace)}$, then $p\in{V(G\lbrace{v}\rbrace)}$ or if $p_1\in{Supp(G-G\lbrace{v}\rbrace)}$, then $p\in{V(G-G\lbrace{v}\rbrace)}$.
We conclude that $p\in{Core(G\lbrace{v}\rbrace)}$ or $p\in{Core(G-G\lbrace{v}\rbrace)}$.

Now, given a vertex $p\in{Core(G\lbrace{v}\rbrace)\bigcup{Core(G-G\lbrace{v}\rbrace)}}$. That is, there is a $p_1\in{Supp(G\lbrace{v}\rbrace)}$ such that $p\in{N(p_1)}$ or there is a $p_2\in{Supp(G-G\lbrace{v}\rbrace)}$ such that $p\in{N(p_2)}$. By item (i) we have that $Supp(G)=Supp(G\lbrace{v}\rbrace)\bigcup{Supp(G-}$ ${G\lbrace{v}\rbrace)}$, then $p_1$ or $p_2\in{Supp(G)}$. Therefore, $p\in{Core(G)}$.

\item Just use the items (i) and (ii).
\item By item (i) we have that $Supp(G)=Supp(G\lbrace{v}\rbrace)\bigcup{Supp(G-G\lbrace{v}\rbrace)}$. Note that $N(Supp(G\lbrace{v}\rbrace))\cap{V(G-G\lbrace{v}\rbrace)}=\emptyset$ and it implies that $N(Supp(G\lbrace{v}\rbrace))\cap{Supp(G-G\lbrace{v}\rbrace)}=\emptyset$. Moreover, by Lemma \ref{tindependent} we have that $Supp(G\lbrace{v}\rbrace)$ and ${Supp(G-G\lbrace{v}\rbrace)}$ are independent sets. Therefore, $Supp(G)$ is an independent set.

\item Just use item (iv).
\end{enumerate}
\end{proof}

\begin{prop}\label{ti4}
Let $G$ be a unicyclic graph of Type $I$, $G\lbrace{v}\rbrace$ its pendant tree such that $v\notin{Supp(G\lbrace{v}\rbrace)}$ and $u,w\in{N(v)\cap{V(G-G\lbrace{v}\rbrace)}}$. If $v\in{Core(G\lbrace{v}\rbrace)}$, for all $y\in\mathcal{N}(G-G\lbrace{v}\rbrace)$ we have that $y_u+y_w=0$ and there is an $x\in\mathcal{N}(G-G\lbrace{v}\rbrace)$ such that $x_u\neq{0}$, then
\begin{enumerate}[label=(\roman*)]
\item $Supp(G)=Supp(G\lbrace{v}\rbrace)\bigcup{Supp(G-G\lbrace{v}\rbrace)}$;
\item $Core(G)=Core(G\lbrace{v}\rbrace)\bigcup{Core(G-G\lbrace{v}\rbrace)}$;
\item $V(\mathcal{G}_N(G))=V(\mathcal{G}_N(G\lbrace{v}\rbrace))\bigcup{V(\mathcal{G}_N(G-G\lbrace{v}\rbrace))}$;
\item $Supp(G)$ is an independent set of $G$;
\item $Supp(G)\cap{Core(G)}=\emptyset.$
\end{enumerate}
\end{prop}
\begin{proof}
\begin{enumerate}[label=(\roman*)]
\item We just use the Theorem \ref{teo11}.
\item Given a vertex $p\in{Core(G)}$. Note that if $p=v$, then $p\in{Core(G\lbrace{v}\rbrace)}$, because $v\in{Core(G\lbrace{v}\rbrace)}$. Assume that $p\in{Core(G)}\setminus\lbrace{v}\rbrace$, then there is a vertex $p_1\in{Supp(G)}$ such that $p\in{N(p_1)}$. Note that ${V(G-G\lbrace{v}\rbrace)\cap{N(V(G\lbrace{v}\rbrace))}}=\lbrace{u,w}\rbrace$ and ${V(G\lbrace{v}\rbrace)\cap{N(V(G-G\lbrace{v}\rbrace))}}=\lbrace{v}\rbrace$.
Since $u,w\in{Supp(G-G\lbrace{v}\rbrace)}$ (because $x_u\neq{0}$ and $x_w\neq{0}$) and $v\notin{Supp(G\lbrace{v}\rbrace)}$ we have that $N(Supp(G\lbrace{v}\rbrace))\cap{V(G-G\lbrace{v}\rbrace)}=\emptyset$ and $N(Supp(G-G\lbrace{v}\rbrace))\cap{V(G\lbrace{v}\rbrace)}=\lbrace{v}\rbrace$.
By item (i) we have that $Supp(G)=Supp(G\lbrace{v}\rbrace)\bigcup{Supp(G-G\lbrace{v}\rbrace)}$, then  $p_1\in{Supp(G\lbrace{v}\rbrace)}$ or $p_1\in{Supp(G-G\lbrace{v}\rbrace)}$. Hence, if $p_1\in{Supp(G\lbrace{v}\rbrace)}$, then $p\in{V(G\lbrace{v}\rbrace)}$ or if $p_1\in{Supp(G-G\lbrace{v}\rbrace)}$, then $p\in{V(G-G\lbrace{v}\rbrace)}$.
We conclude that $p\in{Core(G\lbrace{v}\rbrace)}$ or $p\in{Core(G-G\lbrace{v}\rbrace)}$.

\noindent The inclusion $\supseteq$ is analogous to \ref{ti1} (ii).

\item Just use the items (i) and (ii).
\item Same argument used in Proposition \ref{ti1} (iv).
\item Argument equal to the used in Proposition \ref{ti1} (v).
\end{enumerate}
\end{proof}

\begin{prop}\label{ti6}
Let $G$ be a unicyclic graph of Type $I$, $G\lbrace{v}\rbrace$ its pendant tree such that $v\notin{Supp(G\lbrace{v}\rbrace)}$ and $u,w\in{N(v)\cap{V(G-G\lbrace{v}\rbrace)}}$. If $v\in{V(\mathcal{G}_N(G\lbrace{v}\rbrace)}$, for all $y\in\mathcal{N}(G-G\lbrace{v}\rbrace)$ we have $y_u+y_w=0$ and there is  $x\in\mathcal{N}(G-G\lbrace{v}\rbrace)$ such that $x_u\neq{0}$, then
\begin{enumerate}[label=(\roman*)]
\item $Supp(G)=Supp(G\lbrace{v}\rbrace)\bigcup{Supp(G-G\lbrace{v}\rbrace)}$;
\item $Core(G)=Core(G\lbrace{v}\rbrace)\bigcup{Core(G-G\lbrace{v}\rbrace)}\bigcup{\lbrace{v}\rbrace}$;
\item $V(\mathcal{G}_N(G))=\left(V(\mathcal{G}_N(G\lbrace{v}\rbrace))\setminus{\lbrace{v}\rbrace}\right)\bigcup{V(\mathcal{G}_N(G-G\lbrace{v}\rbrace))}$;
\item $Supp(G)$ is an independent set of $G$;
\item $Supp(G)\cap{Core(G)}=\emptyset.$
\end{enumerate}
\end{prop}
\begin{proof}
\begin{enumerate}[label=(\roman*)]
\item Same argument used in Proposition \ref{ti4} (i).
\item Given a vertex $p\in{Core(G)}$. Note that if $p=v$, then $p\in{\lbrace{v}\rbrace}.$ Assume $p\in{Core(G)}\setminus\lbrace{v}\rbrace$, then there is a vertex $p_1\in{Supp(G)}$ such that $p\in{N(p_1)}$. Note that ${N(V(G\lbrace{v}\rbrace))\cap{V(G-G\lbrace{v}\rbrace)}}=\lbrace{u,w}\rbrace$ and ${N(V(G-G\lbrace{v}\rbrace))\cap{V(G\lbrace{v}\rbrace)}}=\lbrace{v}\rbrace$. Since $u,w\in{Supp(G-G\lbrace{v}\rbrace)}$ (because $x_u\neq{0}$ and $x_w\neq{0}$) and $v\notin{Supp(G\lbrace{v}\rbrace)}$ we have that $N(Supp(G\lbrace{v}\rbrace))\cap{V(G-}$ ${G\lbrace{v}\rbrace)}=\emptyset$ and $N(Supp(G-G\lbrace{v}\rbrace))\cap{V(G\lbrace{v}\rbrace)}=\lbrace{v}\rbrace$. By item (i) $Supp(G)=Supp(G\lbrace{v}\rbrace)\bigcup{Supp(G-}$ ${G\lbrace{v}\rbrace)}$, then $p_1\in{Supp(G\lbrace{v}\rbrace)\bigcup{Supp(G-G\lbrace{v}\rbrace)}}$. Hence, if $p_1\in{Supp(G\lbrace{v}\rbrace)}$, then $p\in{V(G\lbrace{v}\rbrace)}$ or if $p_1\in{Supp(G-G\lbrace{v}\rbrace)}$, then $p\in{V(G-G\lbrace{v}\rbrace)}$, that is, $p\in{Core(G\lbrace{v}\rbrace)}$ or $p\in{Core(G-G\lbrace{v}\rbrace)}$.

Now, given a vertex $p\in{Core(G\lbrace{v}\rbrace)\bigcup{Core(G-G\lbrace{v}\rbrace)}}\bigcup{\lbrace{v}\rbrace}$. Notice that, using (i) we have that $u,w\in{Supp(G)}$, because $u,w\in{Supp(G-G\lbrace{v}\rbrace)}$.
Thus, if $p=v$, then we have that $p\in{Core(G)}$ because $p\in{N(u)}$ and $u\in{Supp(G)}$.  If $p\in{Core(G\lbrace{v}\rbrace)}$, then there exists $r\in{N(p)}$ such that $r\in{Supp(G\lbrace{v}\rbrace)}$. Using the item (i) we can conclude that $r\in{Supp(G)}$. Hence $p\in{Core(G)}$. If $p\in{Core(G-G\lbrace{v}\rbrace)}$, then there exists $r\in{N(p)}$ such that $r\in{Supp(G-G\lbrace{v}\rbrace)}$. Using the item (i) we have that $r\in{Supp(G)}$. Therefore, $p\in{Core(G)}$.

\item Just use the items (i) and (ii).
\item Same argument used in Proposition \ref{ti1} (iv).
\item Argument equal to the used in Proposition \ref{ti1} (v).
\end{enumerate}
\end{proof}

\begin{lemat}\label{lemagv}
Let $G$ be a unicyclic graph of Type $I$ and $G\lbrace{v}\rbrace$ its pendant tree such that $v\notin{Supp(G\lbrace{v}\rbrace)}$. Then $Supp(G\lbrace{v}\rbrace)\subseteq{Supp(G\lbrace{v}\rbrace-v)}.$
\end{lemat}

\begin{proof}
Given $j\in{Supp(G\lbrace{v}\rbrace)}$. It means that there is a vector $x\in\mathcal{N}(G\lbrace{v}\rbrace)$ such that $x_j\neq{0}$.
The adjacency matrix of $G$ can have the following format
\begin{eqnarray*}
A(G\lbrace{v}\rbrace)&=&\begin{blockarray}{cccc}
         & V(G\lbrace{v}\rbrace-v)    & v  \\
\begin{block}{r[ccc]}
V(G\lbrace{v}\rbrace-v) & A(G\lbrace{v}\rbrace-v)  &  \mathbf{a}\\
 v   &   \mathbf{a}^{t}                  &   0    \\
\end{blockarray}.\\
\end{eqnarray*}
Where the submatrix $\mathbf{a}$ has almost all null entries, except for the entries corresponding to the adjacency between $v$ and vertices of the $G\lbrace{v}\rbrace-v$. Note that
\begin{eqnarray*}
x=\left[\begin{array}{cccccccc}
    w  \\
   x_v\\
\end{array}\right],
\end{eqnarray*}

\noindent where $w=\left[\begin{array}{cccccccc}
   x_1 & x_2 & \cdots & x_j & \cdots & x_{v-1}  \\
\end{array}\right]^{t}$ and $x_v=0$, because $v\notin{Supp(G\lbrace{v}\rbrace)}$. The product

\begin{eqnarray*}
A(G\lbrace{v}\rbrace)x=\left[\begin{array}{cccccccc}
    A(G\lbrace{v}\rbrace-v)w+x_v\mathbf{a}\\
   \mathbf{a}^tw\\
\end{array}\right]=\left[\begin{array}{cccccccc}
    A(G\lbrace{v}\rbrace-v)w\\
   \mathbf{a}^tw\\
\end{array}\right]=\mathbf{0}
\end{eqnarray*}
gives $A(G\lbrace{v}\rbrace-v)w=\mathbf{0}$. And, as $(w)_j=x_j\neq{0}$, we obtain that $j\in{Supp(G\lbrace{v}\rbrace-v)}$.
\end{proof}

\begin{prop}
\label{ti3}
Let $G$ be a unicyclic graph of Type $I$, $G\lbrace{v}\rbrace$ its pendant tree such that $v\notin{Supp(G\lbrace{v}\rbrace)}$ and $u,w\in{N(v)\cap{V(G-G\lbrace{v}\rbrace)}}$. If there is a $x\in\mathcal{N}(G-G\lbrace{v}\rbrace)$ such that $x_u+x_w\neq{0}$, then
\begin{enumerate}[label=(\roman*)]
\item $Supp(G)=Supp(G\lbrace{v}\rbrace-v)\bigcup{Supp(G-G\lbrace{v}\rbrace)}$;
\item $Core(G)=Core(G\lbrace{v}\rbrace-v)\bigcup{Core(G-G\lbrace{v}\rbrace)}\bigcup{\lbrace{v}\rbrace}$;
\item $V(\mathcal{G}_N(G))=V(\mathcal{G}_N(G\lbrace{v}\rbrace-v))\bigcup{V(\mathcal{G}_N(G-G\lbrace{v}\rbrace))}$;
\item $Supp(G)$ is an independent set of $G$;
\item $Supp(G)\cap{Core(G)}=\emptyset.$
\end{enumerate}
\end{prop}

\begin{proof}
\begin{enumerate}[label=(\roman*)]
\item As there is a $x\in\mathcal{N}(G-G\lbrace{v}\rbrace)$ such that $x_u+x_w\neq{0}$ and by Lemma \ref{lemagv} $Supp(G\lbrace{v}\rbrace)\subseteq{Supp(G\lbrace{v}\rbrace-v)}$, thus just use the Theorem \ref{teobasistipoi}.
\item Given a vertex $p\in{Core(G)}$. If $p=v$, then $p\in\lbrace{v}\rbrace$. Assume $p\in{Core(G)}\setminus\lbrace{v}\rbrace$. Thus there is a vertex $p_1\in{Supp(G)}$ such that $p\in{N(p_1)}$. Note that, there is no adjacency between $V(G\lbrace{v}\rbrace-v)$ and $V(G-G\lbrace{v}\rbrace)$, that is, $N(V(G\lbrace{v}\rbrace-v))\cap{V(G-G\lbrace{v}\rbrace)}=\emptyset$. Therefore, $N(Supp(G\lbrace{v}\rbrace-v))\cap{V(G-G\lbrace{v}\rbrace)}=\emptyset$. Using the item (i) we can conclude that $p_1\in{Supp(G\lbrace{v}\rbrace-v)}\bigcup{Supp(G-G\lbrace{v}\rbrace)}$.  Hence, if $p_1\in{Supp(G\lbrace{v}\rbrace-v)}$, then $p\in{V(G\lbrace{v}\rbrace-v)}$ or if $p_1\in{Supp(G-}$ ${G\lbrace{v}\rbrace)}$, then $p\in{V(G-G\lbrace{v}\rbrace)}$, that is, $p\in{Core(G\lbrace{v}\rbrace-v)\cup{Core(G-G\lbrace{v}\rbrace)}}$.

Now, given a vertex $p\in{Core(G\lbrace{v}\rbrace-v)\bigcup{Core(G-G\lbrace{v}\rbrace)}}\bigcup{\lbrace{v}\rbrace}$. Note that $x_u+x_w\neq{0}$, then $u\in{Supp(G-G\lbrace{v}\rbrace)}$ or $w\in{Supp(G-G\lbrace{v}\rbrace)}$. Then using the item $(i)$ we can conclude that $u\in{Supp(G)}$ or $w\in{Supp(G)}$. If $p=v$, then $p\in{N(u)\cap{N(w)}}$. Therefore, $p\in{Core(G)}$. Now, if $$p\in{Core(G\lbrace{v}\rbrace-v)}\bigcup{Core(G-G\lbrace{v}\rbrace)},$$
\noindent then there is $p_1\in{Supp(G\lbrace{v}\rbrace-v)}$ such that $p\in{N(p_1)}$ or there is $p_2\in{Supp(G-G\lbrace{v}\rbrace)}$ such that $p\in{N(p_2)}$.  By item (i) we conclude that $p_1\in{Supp(G)}$ or $p_2\in{Supp(G)}$. It implies that $p\in{Core(G)}$.
\item Just use the items (i) and (ii).

\item By item (i) we have that $Supp(G)=Supp(G\lbrace{v}\rbrace-v)\bigcup{Supp(G-G\lbrace{v}\rbrace)}$. Denote $G\lbrace{v}\rbrace-v=\bigcup\limits^{k}_{i=1}{T_i}$, where $T_i$ are the connected components of the forest $G\lbrace{v}\rbrace-v$. Note that $N(Supp(G\lbrace{v}\rbrace-v))\cap{V(G-G\lbrace{v}\rbrace)}=\emptyset$ and it implies that $N(Supp(G\lbrace{v}\rbrace-v))\cap{Supp(G-G\lbrace{v}\rbrace)}=\emptyset$. Moreover, by Lemma \ref{tindependent} we have that $Supp(G\lbrace{v}\rbrace-v)=\bigcup\limits^{k}_{i=1}Supp({T_i})$ and ${Supp(G-G\lbrace{v}\rbrace)}$ are independent sets. Therefore, $Supp(G)$ is an independent set.

\item Argument equal to the used in Proposition \ref{ti1} (v).
\end{enumerate}
\end{proof}

\section{Support, Core and $N$-vertices of Unicyclic Graphs of Type $II$}\label{sec:suptipoII}

In this section, we study the support, core and $N$-vertices of unicyclic graphs of Type $II$. We divide the unicyclic graphs of Type $II$ into two disjoint sets.  Those with a length cycle other than $4t$ (Proposition \ref{ti2}) and those with a length cycle equal to $4t$ (Proposition \ref{ti5}), where $t\in{\mathbb{N}}$.

\begin{prop}\label{ti2}
Let $G$ be a unicyclic graph of Type $II$ and $C$ its cycle. Consider $G-C=\bigcup\limits^{k}_{i=1}{T_i}$, where $T_i$ is a connected component of $G-C$. If $\mid{V(C)}\mid\neq{4t}$ for $t\in{\mathbb{N}}$, then
\begin{itemize}
\item[(i)] $Supp(G)=Supp(G-C)=\bigcup\limits^{k}_{i=1}{Supp(T_i)}$;
\item[(ii)] $Core(G)=Core(G-C)=\bigcup\limits^{k}_{i=1}{Core(T_i)}$;
\item[(iii)] $V(\mathcal{G}_N(G))=V(C)\bigcup\left(\bigcup\limits^{k}_{i=1}{V(\mathcal{G}_N(T_i)})\right)$;
\item[(iv)] $Supp(G)$ is an independent set.
\item[(v)] $Supp(G)\cap{Core(G)}=\emptyset$.
\end{itemize}
\end{prop}

\begin{proof}
\begin{itemize}
\item[(i)] Just use Theorem \ref{t40}.
\item[(ii)] Given a vertex $p\in{Core(G)}$. Then there exists $p_1\in{Supp(G)}$ such that $p\in{N(p_1)}$. By item (i) we have that $p_1\in{Supp(G-C)}$. Using the Lemma \ref{non supp} we can conclude that $N(C)\cap{Supp(G-C)}=\emptyset$. Then $p\in{V(G-C)}$. Therefore, $p\in{Core(G-C)}$.
Now, given a vertex $p\in{Core(G-C)}$. Then there exists $p_1\in{Supp(G-C)}$ such that $p\in{N(p_1)}$.
By item (i) we have that $p_1\in{Supp(G)}$. Hence, $p\in{Core(G)}$.

\item[(iii)] Just use items (i) e (ii).
\item[(iv)] By item $(i)$ we have that $Supp(G)=Supp(G-C)=\bigcup\limits^{k}_{i=1}{Supp(T_i)}$. Note that for all $i\neq{j}$ we have that $N(Supp(T_i))\cap{Supp(T_j)}=\emptyset$. Using the Lemma \ref{tindependent} in every $T_i$ we obtain the result.
\item[(v)] Argument equal to the used in Proposition \ref{ti1} (v).
\end{itemize}
\end{proof}

Lemma \ref{lemagc} is used to proof item (i) of the Proposition \ref{ti5}.
\begin{lemat}\label{lemagc}
Let $G$ be a unicyclic graph and $C$ its cycle. If $G$ is a unicyclic graph of Type $II$, then $Supp(G-C)\subseteq{\bigcup\limits_{v\in{V(C)}}Supp(G\lbrace{v}\rbrace)}.$
\end{lemat}

\begin{proof}
Consider $G-C=\bigcup\limits^{k}_{i=1}{T_i}$, where $T_i$ is a connected component of $G-C$.
We know that $Supp(G-C)=\bigcup\limits_{i=1}^{k}Supp(T_i)$.
Given $w_i\in{Supp(G-C)}$, then there is an $i$ such that $w_i\in{Supp(T_i)}$. Pick $x\in\mathcal{N}(T_i)$ and note that $V(T_i)\subseteq{V(G\lbrace{v}\rbrace)}$ for some $v\in{V(C)}$. We will show that  $x\upharpoonright^{G\lbrace{v}\rbrace}_{T_i}\in{\mathcal{N}(G\lbrace{v}\rbrace)}$. In fact, since $G$ is a unicyclic graph of Type $II$ we have that $v\in{Supp(G\lbrace{v}\rbrace)}$. By Lemma \ref{non supp}, if $u_i\in{V(T_i)}\cap{N(v)}$, then $x_{u_i}=0$.
The adjacency matrix of $G\lbrace{v}\rbrace$ can have the following format

\begin{eqnarray*}
A(G\lbrace{v}\rbrace)&=&\begin{blockarray}{ccccccccc}
         & V(T_i)    & V(G\lbrace{v}\rbrace-T_i)   \\
\begin{block}{r[cccccccc]}
V(T_i)                    & A(T_i)       &  B \\
V(G\lbrace{v}\rbrace-T_i) &  {B}^t      &   A(G\lbrace{v}\rbrace-T_i)   \\
\end{block}
\end{blockarray}.\\
\end{eqnarray*}

\noindent Where $B$ is a submatrix of $A(G\lbrace{v}\rbrace)$ with almost all null entries, except for the entry corresponding to the adjacency between the vertex $u_i$ and the vertex $v$. Then the product

\begin{eqnarray*}
A(G)x\upharpoonright^{G\lbrace{v}\rbrace}_{T_i}=\left[\begin{array}{cccccccc}
    \mathbf{0}\\
    A(T_i)x\\
      B^t.x\\
      \mathbf{0} \\
\end{array}\right]=\left[\begin{array}{cccccccc}
             \mathbf{0}\\
    \mathbf{0}\\
      1.x_{u_i}\\
      \mathbf{0}\\
\end{array}\right]=\mathbf{0}
\end{eqnarray*}

gives $x\upharpoonright^{G\lbrace{v}\rbrace}_{T_i}\in{\mathcal{N}(G\lbrace{v}\rbrace)}$. Therefore, $w_i\in{Supp(G\lbrace{v}\rbrace)}$ and we have that $Supp(G-C)\subseteq{\bigcup\limits_{v\in{V(C)}}Supp(G\lbrace{v}\rbrace)}.$
\end{proof}

Proposition \ref{ti5} tells us what is the support, core and $N$-vertices of a unicyclic graph $G$ of Type $II$, where the cycle of $G$ has length equal to $4t$.
\begin{prop}
\label{ti5}
Let $G$ be a unicyclic graph of Type $II$ and $C$ its cycle. If $\mid{V(C)}\mid={4t}$ for $t\in{\mathbb{N}}$, then
\begin{itemize}
\item[(i)]$Supp(G)=\bigcup\limits_{v\in{V(C)}}Supp(G\lbrace{v}\rbrace);$
\item[(ii)]$Core(G)=V(C)\cup\left(\bigcup\limits_{v\in{V(C)}}Core(G\lbrace{v}\rbrace)\right);$
\item[(iii)] $V(\mathcal{G}_N(G))=\bigcup\limits_{v\in{V(C)}}V(\mathcal{G}_N(G\lbrace{v}\rbrace));$
\item[(iv)]$Supp(G)\cap{Core(G)}=V(C);$
\item[(v)] $Supp(G)$ is not an independent set of $G$.
\end{itemize}
\end{prop}

\begin{proof}
\begin{itemize}
\item[(i)]Since that $G$ is a unicyclic graph of Type $II$ then by Lemma \ref{lemagc} $$Supp(G-C)\subseteq{\bigcup\limits_{v\in{V(C)}}Supp(G\lbrace{v}\rbrace)}.$$ Then using the Theorem \ref{tbasisc4} we obtain the result.
\item[(ii)] Given a vertex $p\in{Core(G)}$, then there exists a vertex $p_1\in{Supp(G)}$ such that $p\in{N(p_1)}$. By item (i) $p_1\in{Supp(G\lbrace{v}\rbrace)}$ for some $v\in{V(C)}$. If $p_1\neq{v}$ then $p\in{V(G\lbrace{v}\rbrace)}$, because $v$ is the only vertex of $G\lbrace{v}\rbrace$ that has neighbors  in $V(G-G\lbrace{v}\rbrace)$. It implies that $p\in {Core(G\lbrace{v}\rbrace)}$.
If $p_1=v$ we have that $p\in{V(C)\cup{V(G\lbrace{v}\rbrace)}}$, therefore $$p\in{V(C)\cup\left(\bigcup\limits_{v\in{V(C)}}Core(G\lbrace{v}\rbrace)\right)}.$$
Given a vertex $p\in{V(C)\cup\left(\bigcup\limits_{v\in{V(C)}}Core(G\lbrace{v}\rbrace)\right)}$. If $p\in{Core(G\lbrace{v}\rbrace)}$ for some pendant tree $G\lbrace{v}\rbrace$, then there exists $p_1\in{Supp(G\lbrace{v}\rbrace)}$ such that $p\in{N(p_1)}$. By item (i)  $p_1\in{Supp(G)}$. Therefore, $p\in{Core(G)}$.
Now, consider $p\in{V(C)}$. Obviously, we have that $p\in{N(v)}$ for some $v\in{V(C)}$. Using the item (i) and since that $G$ is a unicyclic graph of Type $II$ we have that $V(C)\subseteq{Supp(G)}$, then $v\in{Supp(G)}$. Therefore, $p\in{Core(G)}$.

\item[(iii)] Just use items (i) e (ii).
\item[(iv)] Given a vertex $v\in{V(C)}$, then $v\in{Supp(G\lbrace{v}\rbrace)}$ because $G$ is a unicyclic graph of Type $II$. That is, $v\in{\bigcup\limits_{v\in{V(C)}}Supp(G\lbrace{v}\rbrace)}$. By item (i) we have that $v\in{{Supp(G)}}$, and it implies that $V(C)\subseteq{Supp(G)}$. Moreover, there is $w\in{N(v)\cap{V(C)}}$ and since that $w\in{{Supp(G)}}$ we have that $v\in{{Core(G)}}$. Hence, $v\in{{Supp(G)}\cap{Core(G)}}$.
Now, given a vertex $v\in{{Supp(G)}\cap{Core(G)}}$, then there exists a vertex $w\in{Supp(G)}$ such that $w\in{N(v)}$. If $v\in{V(C)}$ we have nothing to proof.  Suppose that $v\notin{V(C)}$. Then we have that $v\in{V(G\lbrace{u}\rbrace))}$ for some $u\in{V(C)}$. Since that $u$ is the only vertex of $G\lbrace{u}\rbrace$ that has neighbors outside of $G\lbrace{u}\rbrace$ we can conclude that $w\in{V(G\lbrace{u}\rbrace)}$. Then by item (i), $w\in{Supp(G\lbrace{u}\rbrace)}$. Note that $v\in{Supp(G\lbrace{u}\rbrace)}$, because
$$v\in{Supp(G)=\bigcup\limits_{v\in{V(C)}}Supp(G\lbrace{v}\rbrace)}.$$
Therefore, $v$ and $w\in{Supp(G\lbrace{u}\rbrace)}$ which is a contradiction, because $Supp(G\lbrace{u}\rbrace)$ is an independent set. Hence, $v\in{V(C)}$.
\item[(v)] By item (i) and since that $G$ is a unicyclic graph of Type $II$, we have that  $V(C)\subseteq{Supp(G)}$. Then $Supp(G)$ is not an independent set.
\end{itemize}
\end{proof}


\section{Null Decomposition of Unicyclic Graphs}
\label{sec:nulldecom}
In this section, we obtain closed formulas for matching and independence numbers of unicyclic graphs. These formulas depend on the support, core and $N$-vertices of unicyclic graphs. Coincidentally, the formulas are very similar to those obtained by Jaume and Molina \cite{tree} for trees.

The next lemmas will be used to prove our main results (theorems \ref{thm1} and \ref{thm2}).

\begin{lemat}\cite{Maikon}
\label{propalpha1}
If $G$ is a unicyclic graph of Type $I$ and $G\lbrace{v}\rbrace$ its pendant tree such that $v\notin{Supp(G\lbrace{v}\rbrace)}$, then $$\alpha{(G)}=\vert{Supp(G\lbrace{v}\rbrace)}\vert+\vert{Supp(G-G\lbrace{v}\rbrace)}\vert{+}\frac{\vert{V(\mathcal{G}_N(G\lbrace{v}\rbrace))}\vert{+}\vert{V(\mathcal{G}_N(G-G\lbrace{v}\rbrace))}\vert}{2}\\.$$
\end{lemat}

\begin{lemat}\cite{Maikon}
\label{propalpha2}
Let $G$ be a unicyclic graph and $C$ its cycle. Consider $G-C=\bigcup\limits^{k}_{i=1}{T_i}$, where $T_i$ is a connected component of $G-C$. If $G$ is a unicyclic graph of Type $II$, then
$$\alpha{(G)}=\floor*{\frac{\vert{V(C)}\vert}{2}}+\sum\limits^{k}_{i=1}\vert{Supp(T_i)}\vert{+}\frac{\vert{V(\mathcal{G}_N(T_i))}\vert}{2}.$$
\end{lemat}

\begin{lemat}\cite{Maikon}
\label{propnu1}
If $G$ is a unicyclic graph of Type $I$ and $G\lbrace{v}\rbrace$ its pendant tree such that $v\notin{Supp(G\lbrace{v}\rbrace)}$, then $$\nu(G)=\vert{Core(G\lbrace{v}\rbrace)}\vert{+}\vert{Core(G-G\lbrace{v}\rbrace)}\vert{+}\frac{\vert{V(\mathcal{G}_N(G\lbrace{v}\rbrace))}\vert+\vert{V(\mathcal{G}_N(G-G\lbrace{v}\rbrace))}\vert}{2}.$$
\end{lemat}

\begin{lemat}\cite{Maikon}
\label{propnu2}
Let $G$ be a unicyclic graph and $C$ its cycle. Consider $G-C=\bigcup\limits^{k}_{i=1}{T_i}$, where $T_i$ is a connected component of $G-C$. If $G$ is a unicyclic graph of Type $II$, then
\begin{equation*}
\nu(G)=\floor*{\frac{\vert{V(C)}\vert}{2}}+\sum\limits^{k}_{i=1}\vert{Core(T_i)}\vert+\frac{\vert{V(\mathcal{G}_N(T_i))}\vert}{2}.
\end{equation*}
\end{lemat}

The next result gives closed formulas for the independence and matching numbers of trees and it will be used to prove lemmas \ref{lematec1} and \ref{lematec3}.

\begin{lemat}\cite{tree}
\label{t35}
Let $T$ be a tree. Then
\begin{eqnarray*}
\nu(T)&=&\vert{Core(T)}\vert+\frac{\vert{V(\mathcal{G}_N(T))}\vert}{2}\\
\alpha(T)&=&\vert{Supp(T)}\vert+\frac{\vert{V(\mathcal{G}_N(T))}\vert}{2}.
\end{eqnarray*}
\end{lemat}

\begin{lemat}\cite{Maikon}
\label{t30}
If $T$ is a tree and $v\in{V(\mathcal{G}_N(T))}$, then there exist $I_1,I_2\in\mathcal{I}(T)$ such that $v\in{I_1}$ and $v\notin{I_2}$.
\end{lemat}

\begin{lemat}\cite{Maikon}\label{tcore}
Let $T$ be a tree and $I$ an independent set of $T$. If $c_i\in{Core(T)}$ and
$c_i\in{I}$, then $I\notin\mathcal{I}(T)$.
\end{lemat}
Lemmas \ref{lematec1} and \ref{lematec3} will be used to prove the Theorem \ref{thm1}.

\begin{lemat}
\label{lematec1}
Let $G$ be a unicyclic graph of Type $I$ and $G\lbrace{v}\rbrace$ its pendant tree such that $v\notin{Supp(G\lbrace{v}\rbrace)}$. Then $$\vert{Supp(G\lbrace{v}\rbrace)}\vert{+}\frac{\vert{V(\mathcal{G}_N(G\lbrace{v}\rbrace))}\vert}{2}=\vert{Supp(G\lbrace{v}\rbrace-v)}\vert{+}\frac{\vert{V(\mathcal{G}_N(G\lbrace{v}\rbrace-v))}\vert}{2}.$$
\end{lemat}
\begin{proof}
Consider $G\lbrace{v}\rbrace-v=\bigcup\limits^{k}_{i=1}{T_i}$, where $T_i$ is a connected component of the $G\lbrace{v}\rbrace-v$ for $i=1,\ldots,k$.
Since $v\notin{Supp(G\lbrace{v}\rbrace})$, then $v\in{V(\mathcal{G}_N(G\lbrace{v}\rbrace))}\cup{Core(G\lbrace{v}\rbrace)}$. Using lemmas \ref{t30} and \ref{tcore} we know that there is a $I\in{\mathcal{I}(G\lbrace{v}\rbrace)}$ such that $v\notin{I}$. Note that $I\in{\mathcal{I}(G\lbrace{v}\rbrace-v)}$. Thus using the Lemma \ref{t35} in each $T_i$ we have
\begin{eqnarray*}
\alpha(G\lbrace{v}\rbrace)=\alpha(G\lbrace{v}\rbrace-v)&=&\sum\limits^{k}_{i=1}\alpha(T_i)=\sum\limits^{k}_{i=1}\vert{Supp(T_i)}\vert{+}\frac{\vert{V(\mathcal{G}_N(T_i))}\vert}{2}\nonumber\\
\vert{Supp(G\lbrace{v}\rbrace)}\vert{+}\frac{\vert{V(\mathcal{G}_N(G\lbrace{v}\rbrace))}\vert}{2}&=&\vert{Supp(G\lbrace{v}\rbrace-v)}\vert{+}\frac{\vert{V(\mathcal{G}_N(G\lbrace{v}\rbrace-v))}\vert}{2}.
\end{eqnarray*}
\end{proof}

\begin{lemat}
\label{lematec3}
Let $G$ be a unicyclic graph of Type $I$ and $G\lbrace{v}\rbrace$ its pendant tree such that $v\notin{Supp(G\lbrace{v}\rbrace)}$. Then $$\vert{Core(G\lbrace{v}\rbrace)}\vert{+}\frac{\vert{V(\mathcal{G}_N(G\lbrace{v}\rbrace))}\vert}{2}=\vert{Core(G\lbrace{v}\rbrace-v)}\vert{+}\frac{\vert{V(\mathcal{G}_N(G\lbrace{v}\rbrace-v))}\vert}{2}+1.$$
\end{lemat}
\begin{proof}
Let $M\in\mathcal{M}(G{\lbrace{v}\rbrace})$. As $v\notin{Supp(G{\lbrace{v}\rbrace})}$, by Lemma \ref{t11} we have that $M$ saturates $v$ . Then there is a vertex $u$ in $G{\lbrace{v}\rbrace}-v$ such that $\lbrace{u,v}\rbrace\in{M}$. First, we will show that $M\setminus{\lbrace{\lbrace{u,v}\rbrace}\rbrace}\in\mathcal{M}(G{\lbrace{v}\rbrace}-v)$. Suppose that $M\setminus{\lbrace{\lbrace{u,v}\rbrace}\rbrace}\notin\mathcal{M}(G{\lbrace{v}\rbrace}-v)$. Hence, there is a matching $M^{'}$ in $G{\lbrace{v}\rbrace}-v$ such that $\vert{M^{'}}\vert{>}\vert{M\setminus{\lbrace{\lbrace{u,v}\rbrace}\rbrace}}\vert$.
 Note that $M^{'}$ is a matching in $G{\lbrace{v}\rbrace}$ as well. As $M\in\mathcal{M}(G{\lbrace{v}\rbrace})$ we have that $\vert{M}\vert\geq\vert{M^{'}}\vert$. Thus
\begin{eqnarray}\label{equa7}
\vert{M}\vert-1<\vert{M^{'}}\vert\leq{\vert{M}\vert}.
\end{eqnarray}
$\vert{M}\vert-1$ and $\vert{M}\vert$ are integer numbers, then equation \eqref{equa7} implies that $\vert{M^{'}}\vert={\vert{M}\vert}$, which is a contradiction. Because $M^{'}$ does not saturate $v$ and all maximum matching in $G{\lbrace{v}\rbrace}$ saturates $v$. Therefore, $M\setminus{\lbrace{\lbrace{u,v}\rbrace}\rbrace}\in\mathcal{M}(G{\lbrace{v}\rbrace}-v)$. Now, we have that $\nu(G\lbrace{v}\rbrace)=1+\nu(G\lbrace{v}\rbrace-v)$. Note that, $G\lbrace{v}\rbrace-v=\bigcup\limits^{k}_{i=1}{T_i}$, where $T_i$'s are the connected components of the forest $G\lbrace{v}\rbrace-v$. Using the Lemma \ref{t35} in each $T_i$ and in $G\lbrace{v}\rbrace$ we obtain that
\begin{eqnarray*}
\nu(G\lbrace{v}\rbrace)&=&1+\nu(G\lbrace{v}\rbrace-v)=1+\sum\limits_{i=1}^{k}\nu(T_i)\\
\vert{Core(G\lbrace{v}\rbrace)}\vert{+}\frac{\vert{V(\mathcal{G}_N(G\lbrace{v}\rbrace))}\vert}{2}&=&1+\sum\limits^{k}_{i=1}\vert{Core(T_i)}\vert{+}\frac{\vert{V(\mathcal{G}_N(T_i))}\vert}{2}\\
\vert{Core(G\lbrace{v}\rbrace)}\vert{+}\frac{\vert{V(\mathcal{G}_N(G\lbrace{v}\rbrace))}\vert}{2}&=&1+\vert{Core(G\lbrace{v}\rbrace-v)}\vert{+}\frac{\vert{V(\mathcal{G}_N(G\lbrace{v}\rbrace-v))}\vert}{2}.
\end{eqnarray*}
\end{proof}

In \cite{Maikon}, we obtained closed formulas for the independence and matching numbers of a unicyclic graph $G$ based on the support of its subtrees. In this paper we also present closed formulas for these parameters of $G$, the difference is that we just use  the support of $G$. Theorems \ref{thm1} and \ref{thm2} give a nice way to compute the independence and matching numbers of unicyclic graphs using its Null Decomposition. More precisely, to compute the independence and matching numbers of unicyclic graphs it is necessary to compute its support, core and $N$ -vertices and verify certain properties of these subsets of vertices.

The following theorem gives closed formulas for the matching number of unicyclic graphs.

\begin{teore}\label{thm1}
Let $G$ be a unicyclic graph. Then
\begin{eqnarray*}
\nu{(G)}&=&\vert{Core(G)}\vert{+}\floor*{\frac{\vert{V(\mathcal{G}_N(G))}\vert-\vert{Supp(G)\cap{Core(G)}}\vert}{2}}.
\end{eqnarray*}
\end{teore}

\begin{proof}
Notice that, since $G$ is a unicyclic graph, then, $G$ must satisfy just the conditions of one of the following propositions: \ref{ti1}, \ref{ti4}, \ref{ti6}, \ref{ti3}, \ref{ti2} or \ref{ti5}. That is, we can divide the unicyclic graphs into six disjoint sets.

\noindent \textbf{Case 1:} Suppose that $G$ satisfies the conditions of Proposition \ref{ti1}.

\noindent By Proposition \ref{ti1} we have that

\begin{eqnarray*}
Core(G)&=&Core(G\lbrace{v}\rbrace)\bigcup{Core(G-G\lbrace{v}\rbrace)},\\
V(\mathcal{G}_N(G))&=&V(\mathcal{G}_N(G\lbrace{v}\rbrace))\bigcup{V(\mathcal{G}_N(G-G\lbrace{v}\rbrace)}\mbox{ and }\\
\emptyset&=&Supp(G)\cap{Core(G)}.
\end{eqnarray*}

Thus, using the Lemma \ref{propnu1}, the matching number of $G$ is given by
\begin{eqnarray*}
\nu(G)&=&\vert{Core(G\lbrace{v}\rbrace)}\vert{+}\vert{Core(G-G\lbrace{v}\rbrace)}\vert{+}\frac{\vert{V(\mathcal{G}_N(G\lbrace{v}\rbrace))}\vert+\vert{V(\mathcal{G}_N(G-G\lbrace{v}\rbrace))}\vert}{2}\\
     &=&\vert{Core(G)}\vert+\frac{\vert{V(\mathcal{G}_N(G))}\vert}{2}=\vert{Core(G)}\vert+\floor*{\frac{\vert{V(\mathcal{G}_N(G))}\vert}{2}}\\
     &=&\vert{Core(G)}\vert+\floor*{\frac{\vert{V(\mathcal{G}_N(G))}\vert-\vert{Supp(G)\cap{Core(G)}}\vert}{2}}.
\end{eqnarray*}

\noindent\textbf{Case 2:} Suppose that $G$ satisfies the conditions of Proposition \ref{ti4}.

This case is analogous to Case 1.

\noindent\textbf{Case 3:}  Suppose that $G$ satisfies the conditions of Proposition \ref{ti6}

Then by Proposition \ref{ti6}, we have that
\begin{eqnarray*}
Core(G)&=&Core(G\lbrace{v}\rbrace)\bigcup{Core(G-G\lbrace{v}\rbrace)}\bigcup{\lbrace{v}\rbrace}\mbox{ and}\\
V(\mathcal{G}_N(G))&=&\left(V(\mathcal{G}_N(G\lbrace{v}\rbrace))\setminus{\lbrace{v}\rbrace}\right)\bigcup{V(\mathcal{G}_N(G-G\lbrace{v}\rbrace))}\\
\emptyset&=&Supp(G)\cap{Core(G)}.
\end{eqnarray*}

And, by Lemma \ref{propnu1}, we have that the matching number of $G$ is the following
\begin{eqnarray*}
\nu(G)&=&\vert{Core(G\lbrace{v}\rbrace)}\vert{+}\vert{Core(G-G\lbrace{v}\rbrace)}\vert{+}\frac{\vert{V(\mathcal{G}_N(G\lbrace{v}\rbrace))}\vert+\vert{V(\mathcal{G}_N(G-G\lbrace{v}\rbrace))}\vert}{2}\\
     &=&\vert{Core(G)}\vert-1+\frac{\vert{V(\mathcal{G}_N(G))}\vert+1}{2}=\vert{Core(G)}\vert+\frac{\vert{V(\mathcal{G}_N(G))}\vert-1}{2}\\
     &=&\vert{Core(G)}\vert{+}\floor*{\frac{\vert{V(\mathcal{G}_N(G))}\vert}{2}}\\
     &=&\vert{Core(G)}\vert+\floor*{\frac{\vert{V(\mathcal{G}_N(G))}\vert-\vert{Supp(G)\cap{Core(G)}}\vert}{2}}.  
\end{eqnarray*}

\noindent\textbf{Case 4:} Suppose that $G$ satisfies the conditions of Proposition \ref{ti3}.

By Proposition \ref{ti3} we have that
\begin{eqnarray*}
Core(G)&=&Core(G\lbrace{v}\rbrace-v)\bigcup{Core(G-G\lbrace{v}\rbrace)\bigcup\lbrace{v}\rbrace},\\
V(\mathcal{G}_N(G))&=&V(\mathcal{G}_N(G\lbrace{v}\rbrace-v))\bigcup{V(\mathcal{G}_N(G-G\lbrace{v}\rbrace))}\mbox{ and }\\
\emptyset&=&Supp(G)\cap{Core(G)}.
\end{eqnarray*}

And, using the Lemmas \ref{propnu1} and \ref{lematec3} we have that the matching number of $G$ is given by

\begin{eqnarray*}
\nu(G)&=&\vert{Core(G\lbrace{v}\rbrace)}\vert{+}\vert{Core(G-G\lbrace{v}\rbrace)}\vert{+}\frac{\vert{V(\mathcal{G}_N(G\lbrace{v}\rbrace))}\vert+\vert{V(\mathcal{G}_N(G-G\lbrace{v}\rbrace))}\vert}{2}\\
&=&\vert{Core(G\lbrace{v}\rbrace-v)}\vert{+}\vert{Core(G-G\lbrace{v}\rbrace)}\vert{+}\frac{\vert{V(\mathcal{G}_N(G\lbrace{v}\rbrace-v))}\vert}{2}\\
&&+\frac{\vert{V(\mathcal{G}_N(G-G\lbrace{v}\rbrace))}\vert}{2}+1\\
&=&\vert{Core(G)}\vert+\frac{\vert{V(\mathcal{G}_N(G))}\vert}{2}=\vert{Core(G)}\vert+\floor*{\frac{\vert{V(\mathcal{G}_N(G))}\vert}{2}}\\
&=&\vert{Core(G)}\vert+\floor*{\frac{\vert{V(\mathcal{G}_N(G))}\vert-\vert{Supp(G)\cap{Core(G)}}\vert}{2}}.
\end{eqnarray*}

\noindent\textbf{Case 5:} Suppose that $G$ satisfies the conditions of Proposition \ref{ti2}.

By Proposition \ref{ti2} we have that
\begin{eqnarray*}
Core(G)&=&Core(G-C)=\bigcup\limits^{k}_{i=1}{Core(T_i)}\mbox{ and }\\
V(\mathcal{G}_N(G))&=&V(C)\bigcup\left(\bigcup\limits^{k}_{i=1}{V(\mathcal{G}_N(T_i)})\right)\\
\emptyset&=&Supp(G)\cap{Core(G)}.
\end{eqnarray*}

And, by Lemma \ref{propnu2}, the matching number of $G$ is the following
\begin{eqnarray*}
\nu{(G)}&=&\floor*{\frac{\vert{V(C)}\vert}{2}}+\sum\limits^{k}_{i=1}\vert{Core(T_i)}\vert{+}\frac{\vert{V(\mathcal{G}_N(T_i))}\vert}{2}\\
   &=&\vert{Core(G)}\vert{+}\floor*{\frac{\vert{V(\mathcal{G}_N(G))}\vert}{2}}\\
   &=&\vert{Core(G)}\vert+\floor*{\frac{\vert{V(\mathcal{G}_N(G))}\vert-\vert{Supp(G)\cap{Core(G)}}\vert}{2}}.       
\end{eqnarray*}

\noindent\textbf{Case 6:} Suppose that $G$ satisfies the conditions of Proposition \ref{ti5}.

By Proposition \ref{ti5}, we know that
\begin{eqnarray*}
Core(G)&=&V(C)\cup\left(\bigcup\limits_{v\in{V(C)}}Core(G\lbrace{v}\rbrace)\right),\\ V(\mathcal{G}_N(G))&=&\bigcup\limits_{v\in{V(C)}}V(\mathcal{G}_N(G\lbrace{v}\rbrace)\mbox{ and}\\ V(C)&=&Supp(G)\cap{Core(G)}.
\end{eqnarray*}
Consider that $G-C=\bigcup\limits^{k}_{i=1}{T_i}$, where the $T_i$'s are the connected components of $G-C$.
Now, note that $\nu(G\lbrace{v}\rbrace)=\nu(G\lbrace{v}\rbrace-v)$, because $v\in{Supp{(G\lbrace{v}\rbrace})}$.
And, using Lemma \ref{propnu2}, we obtain that the matching number is satisfies
\begin{eqnarray*}
\nu(G)&=&\floor*{\frac{\vert{V(C)}\vert}{2}}+\sum\limits^{k}_{i=1}\vert{Core(T_i)}\vert+\frac{\vert{V(\mathcal{G}_N(T_i))}\vert}{2}\\
      &=&\frac{\vert{V(C)}\vert}{2}+\sum\limits^{k}_{i=1}\nu(T_i)=-\frac{\vert{V(C)}\vert}{2}+\vert{V(C)}\vert+\sum\limits^{k}_{i=1}\nu(T_i)\\
      &=&-\frac{\vert{V(C)}\vert}{2}+\vert{V(C)}\vert+\sum\limits_{{v}\in{V(C)}}\nu(G\lbrace{v}\rbrace-v)\\
      &=&-\frac{\vert{V(C)}\vert}{2}+\vert{V(C)}\vert+\sum\limits_{{v}\in{V(C)}}\nu(G\lbrace{v}\rbrace)\\
       &=&-\frac{\vert{V(C)}\vert}{2}+\vert{V(C)}\vert+\sum\limits_{{v}\in{V(C)}}\vert{Core(G\lbrace{v}\rbrace)}\vert+\frac{\vert{V(\mathcal{G}_N(G\lbrace{v}\rbrace))}\vert}{2}\\
       &=&\vert{Core(G)}\vert{+}\frac{\vert{V(\mathcal{G}_N(G))}\vert-\vert{Supp(G)\cap{Core(G)}}\vert}{2}\\
      &=&\vert{Core(G)}\vert{+}\floor*{\frac{\vert{V(\mathcal{G}_N(G))}\vert-\vert{Supp(G)\cap{Core(G)}}\vert}{2}}.
\end{eqnarray*}
\end{proof}

\begin{lemat}\label{lei}
If $T$ is a tree, then $\displaystyle{\bigcap_{I\in\mathcal{I}(T)}I}=Supp(T)$.
\end{lemat}

\begin{proof}
Given a vertex $v\in{\displaystyle{\bigcap_{I\in\mathcal{I}(T)}I}}$. If $v\in{Core(T)\cup{V(\mathcal{G}_N}(T))}$, then by Lemmas \ref{t30} and \ref{tcore} there is a $J\in{\mathcal{I}(T)}$ such that $v\notin{J}$, which is a contradiction, because $v\in{\displaystyle{\bigcap_{I\in\mathcal{I}(T)}I}}$. Therefore, $v\in{Supp(T)}.$

Now, given a vertex $v\in{Supp(T)}$. If  $v\notin{\displaystyle{\bigcap_{I\in\mathcal{I}(T)}I}}$, then there is a maximum independent set $J\in{\mathcal{I}(T)}$ such that $v\notin{J}$. Note that, given $c\in{Core(T)}$ then by Lemma \ref{tcore} we have that $c\notin{J}$. Thus, $J=\lbrace{s_1,s_2,\ldots,s_k}\rbrace\cup\lbrace{n_1,n_2,\ldots,n_r}\rbrace$, where for all $i$ we have that $s_{i}\in{Supp(T)}$ and $n_{i}\in{V(\mathcal{G}_N}(T))$. Therefore, $J\cup{\lbrace{v}\rbrace}$ is an independent set of $T$, because $\lbrace{s_1,s_2,\ldots,s_k}\rbrace\cup\lbrace{v}\rbrace\subseteq{Supp(T)}$ and by Lemma \ref{tindependent} $Supp(T)$ is an independent set of $T$. Moreover, $\vert{J\cup\lbrace{v}\rbrace}\vert=\vert{J}\vert+1>\vert{J}\vert$, which is a contradiction, because $J\in{\mathcal{I}(T)}$. Hence, $v\in{\displaystyle{\bigcap_{I\in\mathcal{I}(T)}I}}$.
\end{proof}

The Lemma \ref{lemac4ti5} will be used to prove the Theorem \ref{thm2}.

\begin{lemat}
\label{lemac4ti5}
Let $G$ be a unicyclic graph and $C$ its cycle. Consider that $G-C=\bigcup\limits^{k}_{i=1}{T_i}$, where the $T_i$'s are the connected components of $G-C$. If $G$ is a unicyclic graph of Type $II$, then
\begin{eqnarray*}
\sum\limits_{v\in{V(C)}}\vert{Supp(G\lbrace{v}\rbrace)}\vert+\frac{\vert{V(\mathcal{G}_N(G\lbrace{v}\rbrace))}\vert}{2}=\vert{V(C)}\vert+\sum\limits^{k}_{i=1}\vert{Supp(T_i)}\vert{+}\frac{\vert{V(\mathcal{G}_N(T_i))}\vert}{2}.
\end{eqnarray*}
\end{lemat}

\begin{proof}
Given $I_v\in\mathcal{I}(G\lbrace{v}\rbrace)$ and since that $G$ is a unicyclic graph of Type $II$ we have that $v\in{Supp(G\lbrace{v}\rbrace)}$. Using the Lemma \ref{lei} we have that $$\displaystyle{Supp(G\lbrace{v}\rbrace)=\bigcap_{I\in\mathcal{I}(G\lbrace{v}\rbrace)}I,}$$
thus $v\in{I_v}$. We will show that $I_v-\lbrace{v}\rbrace\in{\mathcal{I}(G\lbrace{v}\rbrace-v)}$.
Suppose that $I_v-\lbrace{v}\rbrace\notin{\mathcal{I}(G\lbrace{v}\rbrace-v)}$. Therefore, there is an independent set $J$ of $G\lbrace{v}\rbrace-v$ such that
\begin{eqnarray}\label{equa10}
\vert{J}\vert>\vert{I_v-\lbrace{v}\rbrace}\vert=\vert{I_v}\vert-1.
\end{eqnarray}

As $G\lbrace{v}\rbrace-v$ is a subgraph of $G\lbrace{v}\rbrace$, we have that $J$ is also independent set of $G\lbrace{v}\rbrace $. Then

\begin{eqnarray}
\label{equa11}
\vert{J}\vert\leq\vert{I_v}\vert.
\end{eqnarray}
Using \eqref{equa10} and \eqref{equa11} we obtain that $\vert{I_v}\vert\geq\vert{J}\vert>\vert{I_v}\vert-1$. It implies that $\vert{J}\vert=\vert{I_v}\vert$, thus  $J\in\mathcal{I}(G\lbrace{v}\rbrace)$. Which is a contradiction, because $v\notin{J}$ and we know that $v$ is in every maximum independent set of $G\lbrace{v}\rbrace$. Then, we have that $I_v-\lbrace{v}\rbrace\in{\mathcal{I}(G\lbrace{v}\rbrace-v)}$. Note that  $\bigcup\limits_{v\in{V(C)}}G\lbrace{v}\rbrace-v=\bigcup\limits^{k}_{i=1}{T_i}$. By Lemma \ref{t35} we have that $$\alpha{(T_i)}=\vert{Supp(T_i)}\vert{+}\frac{\vert{V(\mathcal{G}_N(T_i))}\vert}{2}\mbox{ and }\alpha(G\lbrace{v}\rbrace)=\vert{Supp(G\lbrace{v}\rbrace)}\vert{+}\frac{\vert{V(\mathcal{G}_N(G\lbrace{v}\rbrace))}\vert}{2}$$
where $i\in\lbrace{1,\ldots,k}\rbrace$ and $v\in{V(C)}$. Then
\begin{eqnarray*}
\sum\limits_{v\in{V(C)}}-1+\alpha(G\lbrace{v}\rbrace)=\sum\limits_{v\in{V(C)}}\alpha(G\lbrace{v}\rbrace-v)&=&\sum\limits^{k}_{i=1}\alpha(T_i)\\
 \sum\limits_{v\in{V(C)}}-1+\sum\limits_{v\in{V(C)}}\alpha(G\lbrace{v}\rbrace)&=&\sum\limits^{k}_{i=1}\alpha(T_i)\\
 -\vert{V(C)}\vert+\sum\limits_{v\in{V(C)}}\vert{Supp(G\lbrace{v}\rbrace)}\vert{+}\frac{\vert{V(\mathcal{G}_N(G\lbrace{v}\rbrace))}\vert}{2}&=&\sum\limits^{k}_{i=1}\vert{Supp(T_i)}\vert{+}\frac{\vert{V(\mathcal{G}_N(T_i))}\vert}{2}.
\end{eqnarray*}
\end{proof}

\begin{teore}\label{thm2}
Let $G$ be a unicyclic graph and $C$ its cycle. If $V(C)\subseteq{V(\mathcal{G}_N(G))}$, then
\begin{eqnarray*}
\alpha{(G)}&=&\vert{Supp(G)}\vert{+}\floor*{\frac{\vert{V(\mathcal{G}_N(G))}\vert}{2}},
\end{eqnarray*}

\noindent otherwise, 

\begin{eqnarray*}
\alpha{(G)}&=&\vert{Supp(G)}\vert{+}\ceil*{\frac{\vert{V(\mathcal{G}_N(G))}\vert-\vert{Supp(G)\cap{Core(G)}}\vert}{2}}.
\end{eqnarray*}
\end{teore}

\begin{proof}
 If $V(C)\subseteq{V(\mathcal{G}_N(G))}$, then $G$ can not satisfy the conditions of propositions \ref{ti4},  \ref{ti3}, \ref{ti6} and \ref{ti5}, because in propositions \ref{ti4}, \ref{ti3}, \ref{ti6} and \ref{ti5} we have that $V(C)\nsubseteq{V(\mathcal{G}_N(G))}$. 
 
\noindent\textbf{Case 1:} Suppose that $G$ satisfies the conditions of Proposition \ref{ti1}.

\noindent By Proposition \ref{ti1} we have that

\begin{eqnarray*}
Supp(G)&=&Supp(G\lbrace{v}\rbrace)\bigcup{Supp(G-G\lbrace{v}\rbrace)}\mbox{ and }\\
V(\mathcal{G}_N(G))&=&V(\mathcal{G}_N(G\lbrace{v}\rbrace))\bigcup{V(\mathcal{G}_N(G-G\lbrace{v}\rbrace)}.
\end{eqnarray*}

Using the Lemma \ref{propalpha1}, we can say that the independence number of $G$ is
\begin{eqnarray*}
\alpha{(G)}&=&\vert{Supp(G\lbrace{v}\rbrace)}\vert+\vert{Supp(G-G\lbrace{v}\rbrace)}\vert{+}\frac{\vert{V(\mathcal{G}_N(G\lbrace{v}\rbrace))}\vert{+}\vert{V(\mathcal{G}_N(G-G\lbrace{v}\rbrace))}\vert}{2}\\
&=&\vert{Supp(G)}\vert{+}\frac{\vert{V(\mathcal{G}_N(G))}\vert}{2}=\vert{Supp(G)}\vert{+}\floor*{\frac{\vert{V(\mathcal{G}_N(G))}\vert}{2}}.
\end{eqnarray*}

\noindent\textbf{Case 2:} Suppose that $G$ satisfies the conditions of Proposition \ref{ti2}.

\noindent By Proposition \ref{ti2} we have that
\begin{eqnarray*}
Supp(G)&=&Supp(G-C)=\bigcup\limits^{k}_{i=1}{Supp(T_i)}\mbox{ and }\\
V(\mathcal{G}_N(G))&=&V(C)\bigcup\left(\bigcup\limits^{k}_{i=1}{V(\mathcal{G}_N(T_i)})\right).
\end{eqnarray*}

Using the Lemma \ref{propalpha2} we have that the independence number of $G$ is given by
\begin{eqnarray*}
\alpha{(G)}&=&\floor*{\frac{\vert{V(C)}\vert}{2}}+\sum\limits^{k}_{i=1}\vert{Supp(T_i)}\vert{+}\frac{\vert{V(\mathcal{G}_N(T_i))}\vert}{2}\\
   &=&\vert{Supp(G)}\vert{+}\floor*{\frac{\vert{V(\mathcal{G}_N(G))}\vert}{2}}.                     
\end{eqnarray*}

Now, suppose that $V(C)\nsubseteq{V(\mathcal{G}_N(G))}$. Note that $G$ can not satisfy the conditions of Proposition  \ref{ti2}, because  in Proposition  \ref{ti2} we have that $V(C)\subseteq{V(\mathcal{G}_N(G))}$.

\noindent\textbf{Case 1:} Suppose that $G$ satisfies the conditions of Proposition \ref{ti1}.

\noindent By Proposition \ref{ti1} we have that

\begin{eqnarray*}
Supp(G)&=&Supp(G\lbrace{v}\rbrace)\bigcup{Supp(G-G\lbrace{v}\rbrace)},\\
V(\mathcal{G}_N(G))&=&V(\mathcal{G}_N(G\lbrace{v}\rbrace))\bigcup{V(\mathcal{G}_N(G-G\lbrace{v}\rbrace)}\mbox{ and }\\
\emptyset&=&Supp(G)\cap{Core(G)}.
\end{eqnarray*}

Using the Lemma \ref{propalpha1}, we can say that the independence number of $G$ is
\begin{eqnarray*}
\alpha{(G)}&=&\vert{Supp(G\lbrace{v}\rbrace)}\vert+\vert{Supp(G-G\lbrace{v}\rbrace)}\vert{+}\frac{\vert{V(\mathcal{G}_N(G\lbrace{v}\rbrace))}\vert{+}\vert{V(\mathcal{G}_N(G-G\lbrace{v}\rbrace))}\vert}{2}\\
&=&\vert{Supp(G)}\vert{+}\frac{\vert{V(\mathcal{G}_N(G))}\vert}{2}=\vert{Supp(G)}\vert{+}\ceil*{\frac{\vert{V(\mathcal{G}_N(G))}\vert}{2}}\\
&=&\vert{Supp(G)}\vert{+}\ceil*{\frac{\vert{V(\mathcal{G}_N(G))}\vert-\vert{Supp(G)\cap{Core(G)}}\vert}{2}}.
\end{eqnarray*}

\noindent\textbf{Case 2:} Suppose that $G$ satisfies the conditions of Proposition \ref{ti4}.

This case is analogous to Case 1.

\noindent\textbf{Case 3:}  Suppose that $G$ satisfies the conditions of Proposition \ref{ti6}

Then by Proposition \ref{ti6}, we have that
\begin{eqnarray*}
Supp(G)&=&Supp(G\lbrace{v}\rbrace)\bigcup{Supp(G-G\lbrace{v}\rbrace)},\\
V(\mathcal{G}_N(G))&=&\left(V(\mathcal{G}_N(G\lbrace{v}\rbrace))\setminus{\lbrace{v}\rbrace}\right)\bigcup{V(\mathcal{G}_N(G-G\lbrace{v}\rbrace))}\mbox{ and }\\
\emptyset&=&Supp(G)\cap{Core(G)}.
\end{eqnarray*}

Using the Lemma \ref{propalpha1} we have that the independence number of $G$ is given by
\begin{eqnarray*}
\alpha{(G)}&=&\vert{Supp(G\lbrace{v}\rbrace)}\vert{+}\frac{\vert{V(\mathcal{G}_N(G\lbrace{v}\rbrace))}\vert}{2}+\vert{Supp(G-G\lbrace{v}\rbrace)}\vert{+}\frac{\vert{V(\mathcal{G}_N(G-G\lbrace{v}\rbrace))}\vert}{2}\\
 &=&\vert{Supp(G)}\vert{+}\ceil*{\frac{\vert{V(\mathcal{G}_N(G\lbrace{v}\rbrace))}\vert-1}{2}}+\frac{\vert{V(\mathcal{G}_N(G-G\lbrace{v}\rbrace))}\vert}{2}\\
           &=&\vert{Supp(G)}\vert{+}\ceil*{\frac{\vert{V(\mathcal{G}_N(G\lbrace{v}\rbrace))\setminus\lbrace{v}\rbrace}\vert}{2}}+\frac{\vert{V(\mathcal{G}_N(G-G\lbrace{v}\rbrace))}\vert}{2}\\
           &=&\vert{Supp(G)}\vert{+}\ceil*{\frac{\vert{V(\mathcal{G}_N(G))}\vert}{2}}\\
           &=&\vert{Supp(G)}\vert{+}\ceil*{\frac{\vert{V(\mathcal{G}_N(G))}\vert-\vert{Supp(G)\cap{Core(G)}}\vert}{2}}.
\end{eqnarray*}

\noindent\textbf{Case 4:} Suppose that $G$ satisfies the conditions of Proposition \ref{ti3}.

By Proposition \ref{ti3} we have that
\begin{eqnarray*}
Supp(G)&=&Supp(G\lbrace{v}\rbrace-v)\bigcup{Supp(G-G\lbrace{v}\rbrace)},\\
V(\mathcal{G}_N(G))&=&V(\mathcal{G}_N(G\lbrace{v}\rbrace-v))\bigcup{V(\mathcal{G}_N(G-G\lbrace{v}\rbrace))}\mbox{ and }\\
\emptyset&=&Supp(G)\cap{Core(G)}.
\end{eqnarray*}

Using the Lemmas \ref{propalpha1} and \ref{lematec1}, we have that the independence number of $G$ is given by
\begin{eqnarray*}
\alpha{(G)}&=&\vert{Supp(G\lbrace{v}\rbrace)}\vert+\vert{Supp(G-G\lbrace{v}\rbrace)}\vert+\frac{\vert{V(\mathcal{G}_N(G\lbrace{v}\rbrace))}\vert+\vert{V(\mathcal{G}_N(G-G\lbrace{v}\rbrace))}\vert}{2}\\
           &=&\vert{Supp(G\lbrace{v}\rbrace-v)}\vert{+}\vert{Supp(G-G\lbrace{v}\rbrace)}\vert{+}\frac{\vert{V(\mathcal{G}_N(G\lbrace{v}\rbrace-v))}\vert}{2}\\
           &&+\frac{\vert{V(\mathcal{G}_N(G-G\lbrace{v}\rbrace))}\vert}{2}\\
           &=&\vert{Supp(G)}\vert{+}\frac{\vert{V(\mathcal{G}_N(G))}\vert}{2}=\vert{Supp(G)}\vert{+}\ceil*{\frac{\vert{V(\mathcal{G}_N(G))}\vert}{2}}\\
           &=&\vert{Supp(G)}\vert{+}\ceil*{\frac{\vert{V(\mathcal{G}_N(G))}\vert-\vert{Supp(G)\cap{Core(G)}}\vert}{2}}.
\end{eqnarray*}

\noindent\textbf{Case 5:} Suppose that $G$ satisfies the conditions of Proposition \ref{ti5}.

By Proposition \ref{ti5}, we know that
\begin{eqnarray*}
Supp(G)&=&\bigcup\limits_{v\in{V(C)}}Supp(G\lbrace{v}\rbrace),\\ 
V(\mathcal{G}_N(G))&=&\bigcup\limits_{v\in{V(C)}}V(\mathcal{G}_N(G\lbrace{v}\rbrace)\mbox{ and}\\ 
V(C)&=&Supp(G)\cap{Core(G)}.
\end{eqnarray*}

Consider that $G-C=\bigcup\limits^{k}_{i=1}{T_i}$, where the $T_i$'s are the connected components of $G-C$. Using the lemmas \ref{propalpha2} and  \ref{lemac4ti5}, we have that the independence number of $G$ is given by
\begin{eqnarray*}
\alpha{(G)}&=&\floor*{\frac{\vert{V(C)}\vert}{2}}+\sum\limits^{k}_{i=1}\vert{Supp(T_i)}\vert{+}\frac{\vert{V(\mathcal{G}_N(T_i))}\vert}{2}\\
           &=&\frac{\vert{V(C)}\vert}{2}+\sum\limits^{k}_{i=1}\vert{Supp(T_i)}\vert{+}\frac{\vert{V(\mathcal{G}_N(T_i))}\vert}{2}\\
           &=&-\frac{\vert{V(C)}\vert}{2}+\vert{V(C)}\vert+\sum\limits^{k}_{i=1}\vert{Supp(T_i)}\vert{+}\frac{\vert{V(\mathcal{G}_N(T_i))}\vert}{2}\\
    &=&-\frac{\vert{Supp(G)\cap{Core(G)}}\vert}{2}+\sum\limits_{v\in{V(C)}}\vert{Supp(G\lbrace{v}\rbrace)}\vert+\frac{\vert{V(\mathcal{G}_N(G\lbrace{v}\rbrace))}\vert}{2}\\
    &=&\vert{Supp(G)}\vert{+}\frac{\vert{V(\mathcal{G}_N(G))}\vert-\vert{Supp(G)\cap{Core(G)}}\vert}{2}\\
    &=&\vert{Supp(G)}\vert{+}\ceil*{\frac{\vert{V(\mathcal{G}_N(G))}\vert-\vert{Supp(G)\cap{Core(G)}}\vert}{2}}.
\end{eqnarray*}
\end{proof}

We refer to the example given in the Introduction of this paper for an illustration of the use of theorems \ref{thm1} and \ref{thm2}. Further, we observe that $I=\lbrace{a,b,e,f,i,j,\ell,m,n,r}\rbrace\in\mathcal{I}(G)$ and $M=\lbrace{\lbrace{a,c}\rbrace,\lbrace{v,j}\rbrace,\lbrace{d,\ell}\rbrace,}$ ${\lbrace{o,m}\rbrace,\lbrace{p,n}\rbrace,\lbrace{r,q}\rbrace,\lbrace{f,g}\rbrace,\lbrace{h,i}\rbrace}\rbrace\in\mathcal{M}(G)$. Because $\vert{I}\vert=10$ and $\vert{M}\vert=8$.

In the following example, we use the theorems \ref{thm1} and \ref{thm2} to obtain $\alpha{(G)}$ and $\nu{(G)}$ of the unicyclic graph $G$ in Figure \ref{fthm2}. Consider $C$ the cycle of $G$.
Computing the support, core and $N$-vertices of $G$ we obtain that $Supp(G)=\lbrace{h,g,i,j,\ell,m,t,u,w,v,y,z}\rbrace$, $Core(G)=\lbrace{f,i,r,s,x}\rbrace$ and $V(\mathcal{G}_N(G))=\lbrace{a,b,c,d,e,m,n,o,p}\rbrace$.

\usetikzlibrary{shapes,snakes}
\tikzstyle{vertex}=[circle,draw,minimum size=1pt,inner sep=1pt]
\tikzstyle{edge} = [draw,thick,-]
\tikzstyle{matched edge} = [draw,snake=zigzag,line width=1pt,-]
\begin{figure}[h!]
\begin{center}
\begin{scriptsize}
\begin{center}
\begin{tikzpicture}[scale=0.7,auto,swap]
\node[draw,star,star points=9,star point ratio=0.6,label=below left:] (1) at (0,0) {$a$};
\node[draw,star,star points=9,star point ratio=0.6,label=below left:] (2) at (1,1) {$b$};
\node[draw,star,star points=9,star point ratio=0.6,label=below left:] (3) at (2,0) {$c$};
\node[draw,star,star points=9,star point ratio=0.6,label=below left:] (4) at (1.7,-1) {$d$};
\node[draw,star,star points=9,star point ratio=0.6,label=below left:] (5) at (0.3,-1) {$e$};
\node[draw,circle,label=below left:] (6) at (1,2) {$f$};
\node[draw,rectangle,label=below left:] (7) at (2,3) {$g$};
\node[draw,rectangle,label=below left:] (8) at (0,3) {$h$};
\node[draw,circle,label=below left:] (9) at (-1,0) {$i$};
\node[draw,rectangle,label=below left:] (10) at (-2,0) {$j$};
\node[draw,rectangle,label=below left:] (11) at (-2,1) {$\ell$};
\node[draw,rectangle,label=below left:] (12) at (-1,1) {$m$};
\node[draw,star,star points=9,star point ratio=0.6,label=below left:] (13) at (3,0) {$n$};
\node[draw,star,star points=9,star point ratio=0.6,label=below left:] (14) at (3,1) {$o$};
\node[draw,star,star points=9,star point ratio=0.6,label=below left:] (15) at (4,0) {$p$};
\node[draw,star,star points=9,star point ratio=0.6,label=below left:] (16) at (5,0) {$q$};
\node[draw,circle,label=below left:] (17) at (3,-1) {$r$};
\node[draw,circle,label=below left:] (18) at (4,-1) {$s$};
\node[draw,rectangle,label=below left:] (19) at (2,-2) {$t$};
\node[draw,rectangle,label=below left:] (20) at (5,-2) {$u$};
\node[draw,rectangle,label=below left:] (21) at (3,-2) {$v$};
\node[draw,rectangle,label=below left:] (22) at (4,-2) {$w$};
\node[draw,circle,label=below left:] (23) at (-1,-2) {$x$};
\node[draw,rectangle,label=below left:] (24) at (-1,-1) {$y$};
\node[draw,rectangle,label=below left:] (25) at (-2,-1) {$z$};
\node at (1,-3) {$G$};
\foreach \from/\to in {1/2,2/3,3/4,4/5,5/1,6/2,6/7,6/8,9/1,9/10,9/11,9/12,13/14,13/3,13/15,15/16,4/17,17/18,17/19,17/21,18/20,18/22,5/23,23/24,23/25} {
 \draw (\from) -- (\to);}
 \foreach \source / \dest in {1/2,3/4,6/7,13/14,15/16,17/21,18/20,23/24,9/12}
   \path[matched edge] (\source) -- (\dest);
\end{tikzpicture}
\caption{Unicyclic graph $G$ its support.}\label{fthm2}
\end{center}
\end{scriptsize}
 \end{center}
\end{figure}
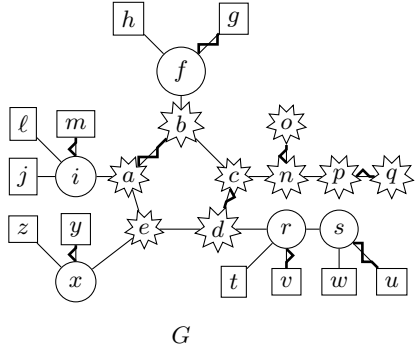

Observe that $V(C)\subseteq{V(\mathcal{G}_N(G))}$. Therefore, by theorems \ref{thm1} and \ref{thm2}, the independence and matching numbers of $G$ are given by
\begin{eqnarray*}
\alpha{(G)}&=&\vert{Supp(G)}\vert{+}\floor*{\frac{\vert{V(\mathcal{G}_N(G))}\vert}{2}}=11+\floor*{\frac{9}{2}}=15\\
\nu{(G)}&=&\vert{Core(G)}\vert{+}\floor*{\frac{\vert{V(\mathcal{G}_N(G))}\vert-\vert{Supp(G)\cap{Core(G)}}\vert}{2}}=5+\floor*{\frac{9-0}{2}}=9.
\end{eqnarray*}

Notice that $I=\lbrace{g,h,j,\ell,m,u,v,w,t,z,y,a,c,o,q}\rbrace\in\mathcal{I}(G)$ and  $M=\lbrace{\lbrace{a,b}\rbrace,\lbrace{d,c}\rbrace,}$ ${\lbrace{f,g}\rbrace,\lbrace{n,o}\rbrace,\lbrace{p,q}\rbrace,\lbrace{r,v}\rbrace,\lbrace{s,u}\rbrace,\lbrace{x,y}\rbrace,\lbrace{m,i}\rbrace}\rbrace\in\mathcal{M}(G)$. Because $\vert{I}\vert=15$ and $\vert{M}\vert=9$.

In the next example, we use the theorems \ref{thm1} and \ref{thm2} to obtain $\alpha{(G)}$ and $\nu{(G)}$ of the unicyclic graph $G$ in Figure \ref{fthm4}. Consider $C$ the cycle of $G$. Computing the support, core and $N$-vertices of $G$ we obtain that $Supp(G)=\lbrace{b,c,d,e,j,\ell,u,v,z,w}\rbrace$, $Core(G)=\lbrace{f,a,i,u,v,z,w}\rbrace$ and $V(\mathcal{G}_N(G))=\lbrace{g,h}\rbrace$.

\usetikzlibrary{shapes,snakes}
\tikzstyle{vertex}=[circle,draw,minimum size=1pt,inner sep=1pt]
\tikzstyle{edge} = [draw,thick,-]
\tikzstyle{matched edge} = [draw,snake=zigzag,line width=1pt,-]
\begin{figure}[h!]
\begin{scriptsize}
\begin{center}
\begin{tikzpicture}[scale=1.2,auto,swap]
\node[draw,rectangle,label=below left:] (1) at (0,0) {$u$};
\node[draw,rectangle,label=below left:] (2) at (1,0) {$v$};
\node[draw,rectangle,label=below left:] (3) at (1,1) {$w$};
\node[draw,rectangle,label=below left:] (4) at (0,1) {$z$};
\node[draw,circle,label=below left:] (5) at (-1,1.5) {$a$};
\node[draw,rectangle,label=below left:] (6) at (-2,1) {$b$};
\node[draw,rectangle,label=below left:] (7) at (2,1) {$c$};
\node[draw,rectangle,label=below left:] (8) at (3,1) {$d$};
\node[draw,rectangle,label=below left:] (9) at (4,1) {$e$};
\node[draw,circle,label=below left:] (10) at (2,0) {$f$};
\node[draw,star,star points=9,star point ratio=0.6,label=below left:] (11) at (3,0) {$g$};
\node[draw,star,star points=9,star point ratio=0.6,label=below left:] (12) at (4,0) {$h$};
\node[draw,circle,label=below left:] (13) at (-1,0.5) {$i$};
\node[draw,rectangle,label=below left:] (14) at (-2,0) {$j$};
\node[draw,rectangle,label=below left:] (15) at (-1,-0.5) {$\ell$};
\foreach \from/\to in {1/2,2/3,3/4,4/1,5/6,5/4,10/7,10/8,10/9,10/3,10/11,11/12,13/1,13/14,13/15}{
\draw (\from) -- (\to);}
\foreach \source / \dest in {5/6,9/10,13/14,11/12,1/2,3/4}
   \path[matched edge] (\source) -- (\dest);
\end{tikzpicture}
\end{center}
\end{scriptsize}
 \caption{ Unicyclic graph $G$ and its support.}
\label{fthm4}
\end{figure}
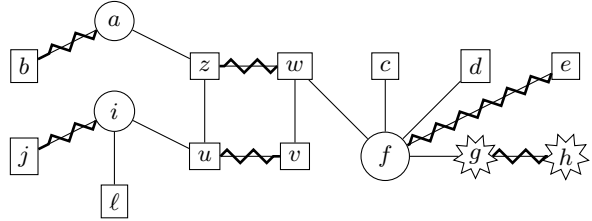

Notice that  $V(C)\nsubseteq{V(\mathcal{G}_N(G))}$. Therefore, by theorems \ref{thm1} and \ref{thm2}, we have that the independence and matching numbers of $G$ are given by
\begin{eqnarray*}
\alpha{(G)}&=&\vert{Supp(G)}\vert{+}\ceil*{\frac{\vert{V(\mathcal{G}_N(G))}\vert-\vert{Supp(G)\cap{Core(G)}}\vert}{2}}=10+\ceil*{\frac{2-4}{2}}=9\\
\nu{(G)}&=&\vert{Core(G)}\vert{+}\floor*{\frac{\vert{V(\mathcal{G}_N(G))}\vert-\vert{Supp(G)\cap{Core(G)}}\vert}{2}}=7+\floor*{\frac{2-4}{2}}=6.
\end{eqnarray*}
Observe that $I=\lbrace{b,c,d,e,g,j,\ell,v,z}\rbrace\in\mathcal{I}(G)$ and $M=\lbrace{\lbrace{a,b}\rbrace,\lbrace{e,f}\rbrace,\lbrace{i,j}\rbrace,\lbrace{g,h}\rbrace,}$ ${\lbrace{u,v}\rbrace,\lbrace{z,w}\rbrace}\rbrace\in\mathcal{M}(G)$. Because $\vert{I}\vert=9$ and $\vert{M}\vert=6$.

\section*{Acknowledgments} Work supported by MATHAMSUD 18-MATH-01.
Maikon Toledo thanks CAPES for their support. V. Trevisan acknowledges
partial support of CNPq grants 409746/2016-9 and 303334/2016-9, and FAPERGS (Proj.\ PqG 17/2551-0001).

\end{document}